\documentclass[times,doublespace]{article}

\usepackage[utf8]{inputenc}

\usepackage[T1]{fontenc}
\usepackage[english]{babel}

\usepackage{graphicx,psfrag}

\usepackage[margin=1.2in]{geometry}
\usepackage{amssymb,amsmath,amsthm,enumerate}

\usepackage{tikz}

\usepackage[colorlinks=true,linkcolor=blue,citecolor=red,urlcolor=blue,pdfborder={0 0 0},pagebackref=true]{hyperref}

\usepackage{mathrsfs}
\usepackage{stmaryrd}      
\usepackage{MnSymbol}      
\usepackage{bbm}           

\usepackage{macros}

\title{Markovian linearization of random walks on groups}
\author{Charles Bordenave\thanks{Institut de Math\'ematiques de Marseille, CNRS \& Aix-Marseille University, Marseille, France.}  \and Bastien Dubail\thanks{Département d'Informatique de l'ENS, ENS, CNRS, PSL University, France and  INRIA, Paris, France. }}

\begin{document}
	
\maketitle

\begin{abstract}
	In operator algebra, the linearization trick is a technique that reduces the study of a non-commutative polynomial evaluated at elements of an algebra $\cA$ to the study of a polynomial of degree one, evaluated  on the enlarged algebra $\cA \otimes M_r (\bC)$, for some integer $r$. 
	We introduce a new instance of the linearization trick which is tailored to study a finitely supported random walk $G$ by studying instead a nearest-neighbour coloured random walk on $G \times \{1, \ldots, r \}$, which is much simpler to analyze. As an application we extend well-known results for nearest-neighbour walks on free groups and free products of finite groups to coloured random walks, thus showing how one can obtain new results for finitely supported random walks, namely an explicit description of the harmonic measure and formulas for the entropy and drift. 
\end{abstract}

\maketitle

\section{Introduction}

Let $\cA$ be a complex unital algebra and consider a non-commutative polynomial $\cP=\cP(x_{1}, \ldots, x_{n})$ in the variables $x_{1}, \ldots, x_{n} \in \cA$. In many cases, a detailed study of the relevant properties of $\cP$ is possible only when the degree of $\cP$ is small, typically when $\cP$ is of degree one, in which case $\cP$ is just a linear combination of the $x_{i}$'s. The linearization trick precisely consists in constructing another polynomial $\tilde{\cP}$ of degree one, which can be related to the relevant properties of $\cP$. It thus can be used to make computations that were not a priori possible for $\cP$. The price to pay is an  enlargement of the algebra: writing $\cP$ as a finite sum of monomials
\[
 \cP = \sum_{i_{1}, \ldots , i_{k}} \alpha_{i_{1}\ldots i_{k}} \: x_{i_{1}} \cdots \ x_{i_{k}},
\]
with $\alpha_{i_{1} \ldots i_{k}} \in \bC$, $\tcP$ is generally constructed as 
\[
 \tilde{\cP} = \sum_{i} \tilde{\alpha}_{i} \otimes x_{i}
\]
where $\tilde{\alpha}_{i} \in M_{r}(\bC)$ are complex matrices. Therefore $\tilde{\cP}$ is no longer in the algebra $\cA$ but in the algebra $M_{r}(\bC) \otimes \cA$ for some integer $r$ that depends on $\cP$ (usually through its degree).

Under various names, the linearization trick has been used in several domains such as electrical engineering, random matrices, operators algebras, automata theory, etc, we refer to the survey \cite{MR3718048}. A famous application of the linearization trick in operator algebra is the computation of the spectrum of a non-commutative polynomial. If $a \in \cA$, the spectrum of $a$ is defined as $\sigma(a) := \{ z \in \bC, z I - a \text{ is not invertible} \}$. Given a non-commutative polynomial $\cP(x_{1}, \ldots, x_{n}) \in \cA$, it is possible to construct a linear $\tilde{\cP} \in M_{r}(\bC) \otimes \cA$ for some integer $r$, such that $(z I -\cP)$ is invertible in $\cA$ if and only if $(\Lambda(z)\otimes I_{\cA} - \tilde{\cP})$ is invertible in $M_{r}(\bC) \otimes \cA$, where $\Lambda(z)$ is the $r \times r$ matrix with only zero elements except the $(1,1)$ entry which is $z$. Moreover, $(z I -\cP)^{-1}$ is then precisely the $(1,1)$ entry of $(\Lambda(z) \otimes I_{\cA} -\tilde{\cP})^{-1}$ (seen as a $r\times r$ matrix of elements in $\cA$). Such a construction can be found for instance in the monograph \cite[Chapter 10]{mingo2017free}. 

As illustrated in this last example, we note that the relevant properties of $\cP$ (here its spectrum and its resolvent) dictates the linearization procedure. In this paper, we introduce a new linearization of the Markov kernel of a random walk with finite range on a group. This linearization is a new Markov kernel on an enlarged state space and it corresponds to a nearest-neighbour random walk on the group. Many properties of the original random walk can be read on the new random walk. We illustrate our method on free groups and free product of finite groups. Notably, we give a new description of the harmonic measure for finite range random walks on these groups and obtain formulas for the drift and entropy as a byproduct.

\subsection{The linearization trick and coloured random walks on groups}

We let $G$ be a finitely generated group with identity element $e$ and we fix a finite set of generators $S$. We assume that $S$ is symmetric, meaning that $g \in S$ implies $g^{-1} \in S$. Let $p=(p_{g})_{g \in G}$ be a probability measure on $G$. Consider the Markov chain $(X_{n})_{n \geq 0}$ on $G$ with transition probabilities
\[
 \bP[X_{n}=h \: | \: X_{n-1} = g] = p_{g^{-1}h},
\]
for all $g,h \in G$. Such a Markov chain is called a convolution random walk.

The random walk is said to be finitely supported, or to have finite range, if the measure $p$ is finitely supported. It is a nearest-neighbour random walk if $p$ is supported on the set of generators $S$. %The simple random walk is the nearest-neighbour random walk with uniform distribution on the set $S$.

As for any Markov chain, the transition kernel of a right convolutional random walk can be seen as an operator acting on square integrable functions: $f=f(x) \mapsto (\cP f)(x)=\sum_{y} P(x,y) f(y)$. In the present case, it can be written in terms of the right multiplication operators $\rho(g)$, $g \in  G$, defined as: for all $f \in \ell^{2}(G)$, 
\[
 \rho(g) \cdot f : x \mapsto f(xg).
\]
Letting $\cP(g,h)=p_{g^{-1}h}$ it is then possible to write 
\begin{equation}\label{eq:defcP}
 \cP = \sum_{g \in G} p_{g} \rho(g).
\end{equation}
This sum is finite if and only if the random walk is finitely supported. Furthermore, $\rho$ is a group morphism so $\rho(gh) = \rho(g) \rho(h)$ for all $g,h \in G$ ($\rho$ is the right regular representation). Writing each element $g$ in the support of the walk as a product of elements of $S$, it is thus possible to write $\cP$ as a non-commutative polynomial in the operators $\rho(g), g \in S$. In other words, $\cP$ is an element of the left group algebra which is generated by the $\rho(g)$'s, $g \in G$. The polynomial $\cP$  is of degree $1$ if and only if the random walk is nearest neighbour. With this point of view, the linearization trick potentially allows to study any finitely supported random walk by considering instead a ``nearest-neighbour operator''
\begin{equation}\label{eq:deftcP}
 \tilde{\cP} = \sum_{g \in G} \tp_{g} \otimes \rho(g)
\end{equation}
for some well-chosen matrices $\tp_{g}\in M_r(\bC)$ such that 
$\sum_{g} \tp_{g}$ is a stochastic matrix, that is, the transition kernel of a Markov chain on $[r]$ where for $n \geq 1$ integer, we set $[n] = \{1, \ldots, n\}$. 
%For all $m,n \in \bZ$, let $[m:n]$ denote the set of integers between $m$ and $n$. When $m = 1$, 

Then $\tcP$  is the transition kernel of a Markov chain $(\tX_{n})_{n}$ on  the state space $G \times [r]$, which at each step moves from a pair $(g,u)$ to a pair $(h,v)$ with probability $\tp_{g^{-1}h}(u,v)$. Equivalently, if one interprets $[r]$ as a set of colours, $\tX_{n}$ can at each step be multiplied by $g \in G$ while the colour is updated from $u$ to $v$ with probability $p_{g}(u,v)$. Such a Markov chain will be called a coloured random walk and can thus be defined similarly to a classical convolution random walk. We note that such random walks already appeared under different names: random walks with internal degrees of freedom \cite{kramli1983random,kaimanovich2005munchhausen,MR2762673}, covering Markov operators \cite{kaimanovich1995poisson}, matrix-valued random walks \cite{connes1989hyperfinite}. We note also that a similar idea of using coloured random walks to study finite range random walks on groups has already been considered by Jean Mairesse (personal communication).

\begin{Def}\label{def:coloured}
  Let $p=(p_{g})_{g \in G}$ be a family of matrices with non-negative coefficients of $M_{r}(\bR)$ such that 
\begin{equation}\label{eq:defP}
P = \sum_{g \in G} p_{g}
\end{equation}  is a stochastic matrix. A coloured random walk is a Markov chain $(Y_{n})_{n}$ on the state space $G \times [r]$, with transition probabilities 
 \[
 \bP[Y_{n} = (h,v) \: | \: Y_{n-1} = (g,u)] = p_{g^{-1}h}(u,v).
 \]
 The support of the random walk is the set of elements $g \in G$ such that $p_{g} \neq 0$. The coloured walk is finitely supported if the support is finite, nearest-neighbour if the support is included in $S$.
\end{Def}

By looking only at the colour coordinate, $(Y_{n})$ induces a Markov chain on the set of colours $[r]$ whose transition matrix is exactly the matrix $P$ in \eqref{eq:defP}.

\paragraph{Irreducibility assumptions.}
For standard convolution random walks, it is generally assumed that the support generates the whole group, which then makes the random walk irreducible. For coloured random walks, we will make the same assumption but suppose as well that the matrix $P$ defines an irreducible Markov chain on the colours. Note this assumption does not necessarily make the coloured walk irreducible, for there may be  unreachable pairs $(g,u)$. 

\begin{Def}
 A coloured walk is quasi-irreducible if the marginal of its support on $G$ generates the whole group and the matrix $P$ is irreducible.
\end{Def}

\paragraph{Reversibility.} An important property satisfied by some Markov chains is reversibility. Recall that, if $X$ is a countable set, $\nu$ a measure on $X$ and $Q$ is a Markov chain, then $Q$ is said to be reversible with respect to $\nu$ if for all $x,y \in X$, $\nu(x) Q(x,y) = \nu(y) Q(y,x)$. In other words, $Q$ seen as an operator on $\ell^2(X,\nu)$ is self-adjoint. In our setting, $\cP$ defined in \eqref{eq:defcP} is said to be reversible if it is reversible for the counting measure on $G$. This is equivalent to the condition: for all $g \in G$, $p_g = p_{g^{-1}}$. Similarly, consider a coloured random walk  $\tcP$ of the form \eqref{eq:deftcP}. Assume that $P$ defined by \eqref{eq:defP} has invariant measure $\pi$ on $[r]$. Then $\tcP$ is reversible if it is reversible for the product of $\pi$ and the counting measure on $G$. This is equivalent to the condition, for all $g \in G$, $u,v \in [r]$, 
$$
\pi(u) p_{g} (u,v) = \pi(v) p_{g^{-1}} (v,u).
$$

\subsection{Main results}

\paragraph{Linearizing random walks.}

The first main contribution of this paper is to formalize a way to apply the linearization trick to random walks on groups.

\begin{Def}\label{def:linear}
 Let $G$ be a a group generated by a finite set $S$, with identity element $e$. Let $(X_{n})_{ n \geq 0}$  be a random walk with kernel $\cP$ as in \eqref{eq:defcP} and with finite support generating $G$. Let $(Y_{n})_{n \geq 0}$ on $G \times [r]$ be a quasi-irreducible, nearest-neighbour coloured random walk with kernel $\tcP$ as in \eqref{eq:deftcP}. We say that  $(Y_{n})_{n \geq 0}$  linearizes $(X_{n})_{n \geq 0}$ (or $\tcP$ linearizes $\cP$) if the following two property holds: $(i)$ if $Y_{0}=(e,1)$ there exists a sequence of stopping times $(\tau_{n})_{n \geq 0}$ with $\tau_{0} = 0$ such that $(Y_{\tau_{n}})_{n \geq 0}$ is a realization of the random walk $(X_n)_{n\geq 0}$ with initial condition $e$, and $(ii)$  these stopping times are a renewal process, that is the variables $\tau_{n+1} - \tau_{n}$ are iid with finite mean. 
 \end{Def}

To be precise, in the above definition, when saying that $Y_{\tau_{n}}$ is a convolution random walk, we of course identify $Y_{\tau_{n}}$ with its $G$ coordinate and forget about the colour, which in this case is constant equal to $1$. %We will often make this abuse of notation when the colour is of little importance, even when the colour is not constant.

We remark also that due to the invariance of transition probabilities by translation by elements of $G$ one can always translate the starting point to $e \in G$, so there is no loss of generality in supposing that the walks are started at $e \in G$.

\begin{thm}\label{thm:linearization_nonrev}
 Let $G$ be a a group generated by a finite set $S$, with identity element $e$. Consider a random walk  with kernel $\cP$ as in \eqref{eq:defcP}  with finite support generating $G$. Then there exists $r \geq 1$ and a coloured random walk on $G \times [r]$ with kernel $\tcP$ as in \eqref{eq:deftcP} which linearizes $\cP$. Moreover, if $\cP$ is reversible, then $\tcP$ can be chosen to be also reversible. 
\end{thm}

Theorem \ref{thm:linearization_nonrev} has been stated in a non-constructive manner for ease of notation. The proof of Theorem \ref{thm:linearization_nonrev} in Section \ref{section:linearization} will exhibit two simple linearization constructions which have a vector $p = (p_g)_{g \in G}$ as input and gives as output the integer $r$ and the family of matrices $(\tilde p_g)_{g \in S}$. There is one construction in the general case and one construction which preserves reversibility. We refer to Remark \ref{rk:bdr} for the number of colours $r$ needed in both constructions.
We note also that the previous spectral linearization tricks reviewed in  \cite{MR3718048} did not preserve the Markov property and could not be used to prove Theorem \ref{thm:linearization_nonrev}.

There are possible extensions of the Markovian linearization trick. It is  possible with the same proof techniques to linearize a coloured random walk on $G \times [r]$ with finite range and obtain a coloured nearest-neighbour coloured random walk on $G \times [r']$ with $r' \geq r$.  
In another direction, it is possible to linearize random walks on $G$ which do not have a finite range provided that we allow a countable set of colours, see Remark \ref{rk:infiniteR} below.  Finally, in this paper we focus on groups only but our first linearization construction applies to random walks on monoids as well. 

%
%
%
%Mimicking what was done with the spectral linearization, we also inquired about a construction of linearized coloured walks that would preserve reversibility. Recall that a Markov chain with transition matrix $Q$ is reversible if there exists a measure $\nu$ such that $\nu (x) Q(x,y) = \nu(y) Q(x,y)$ for all $x,y$. When it comes to convolution random walks, this amounts to the fact that $p_{g} = p_{g^{-1}}$, for all $g \in G$.
%For quasi-irreducible coloured random walks, the requirement for reversibility is that
%\begin{equation}
% \pi(u) p_{g} (u,v) = \pi(v) p_{g^{-1}}(v,u) \quad \forall u,v \in [r].
%\end{equation}
%where $\pi$ is the unique invariant probability measure of the matrix $P=\sum_{g} p_{g}$ (uniqueness is ensured by quasi-irreducibility). 
%
%% Letting $\innerprod{\cdot}{\cdot}_{\pi}$ denote the innerproduct on $l^{2}(G \times [r])$:
%% \[
%%  \innerprod{f}{h}_{\pi} := \sum_{g \in G, u \in [r]} \pi(u) f(g,u) h(g,u)
%% \]
%% this amounts to saying that $p_{g^{-1}}=p_{g}^{\ast}$ is the adjoint of $p_{g}$, or equivalently $\tilde{\cP}$ is self-adjoint for the natural extension of $\ast$ to operators on $l^{2}(G \times [r])$.
%
%\begin{thm}\label{thm:linearization_rev}
%  Let $G$ be a a group generated by a finite set $S$, with identity element $e$. Consider a probability measure $p$ with finite support generating the whole group. If $p_{g} = p_{g^{-1}}$ for all $g \in G$, then there exists a reversible coloured random walk that linearizes the convolution walk defined by $p$.
%\end{thm}

\paragraph{Explicit description of the harmonic measure}

The whole point of the linearization trick is that the complexity added by passing to matrices is expected to be much smaller than the benefit of having nearest-neighbour random walks. The second part of this paper provides such an example where results extend naturally to the matrix setting, thus illustrating the use of the linearization trick. Namely, we extend the explicit description of the harmonic measure for nearest-neighbour convolution walks on groups with an underlying tree structure, such as free groups \cite{kaimanovich1983random, ledrappier2001some} and free products of finite groups \cite{jean2005random,mairesse2007random, MR2378433, gilch2011asymptotic}. Following \cite{jean2005random}, these two cases can be considered at once. 

Write $\Fd$ for the free group on $d$ generators $a_{1}, \ldots a_{d}$ and let $G_{1}, \ldots, G_{m}$ be finite groups. Consider the group $G = \Fd \ast G_{1} \ast \cdots \ast G_{m}$, with the set of generators
\begin{equation}\label{eq:generators}
S := \bigcup_{i=1}^{d} \{a_{i}, a_{i}^{-1}\} \bigcup \left( \bigsqcup_{j=1}^{m} S_{j} \right).
\end{equation}
In the above expression, for all $j = 1, \ldots, m$, $S_{j} := G_{j} \setminus \{e\}$ is the set of elements of $G_{j}$ distinct from the identity element, and $\bigsqcup$ denotes a disjoint union. Such groups were considered in \cite{jean2005random,haring1983groups} under the name \emph{plain groups}.

Introducing the notation from \cite{jean2005random},
\begin{equation}\label{eq:next}
	\forall g \in S, \quad \Next(g) := \left\{ \begin{array}{l l}
		S \setminus \{g^{-1}\} & \text{if $g \in \Fd$} \\
		S \setminus S_{i} & \text{if $g \in S_{i}$}
	\end{array}\right.,
\end{equation}
we see that every element $g \in G$ writes uniquely as a word $g=g_{1} \cdots g_{n}$ with $n = \abs{g}$ and $g_{i+1} \in \Next(g_{i})$ for all $i=1, \ldots, n-1$.

\smallskip

Consider an irreducible convolution random walk $(X_{n})_{\geq 0}$ on $G$. If we set aside the trivial cases where $G$ is isomorphic to $\bZ$ or $\bZ / 2 \bZ \ast \bZ / 2 \bZ$, $G$ is a non-amenable group which implies that $(X_{n})_{n \geq 0}$ is transient (see for instance \cite[Chapter 1]{woess2000random}). More precisely, arguments going back to Furstenberg \cite{MR163345} show that $X_{n}$ converges a.s. to an infinite word $X_{\infty} = (\xi_{i})_{i \geq 1}$ with the property that for all $i \geq 1, \xi_{i} \in S$ and $\xi_{i+1} \in \Next(\xi_{i})$. The law $p^{\infty}$ of $X_{\infty}$ is called the harmonic measure and it provides much information on the asymptotic properties of $X_{n}$. 

In the case $(X_{n})_{n \geq 0}$ is nearest-neighbour, the harmonic measure can be described very explicitely. In particular it is Markovian: writing $X_{\infty}^{(k)}$ for the truncature of $X_{\infty}$ to its $k$ first letters, the sequence $(X_{\infty}^{(k)})_{k \geq 1}$ is a Markov chain. The transition probabilities are as follows: given $g \in S$, let $\mu_{g} := p^{\infty}\sbra{\xi_{1} = g} = \bP \sbra{X_{\infty}^{(1)} = g}$. Then for all $g_{1}, \ldots, g_{k} \in S$, satisfying $g_{i+1} \in \Next(g_i)$ for all $i$, 
\begin{equation}\label{eq:harmonic_nocolour}
	p^{\infty} \sbra{ (\xi_i)_{i=1}^{k} = (g_{i})_{i=1}^{k} } = \bP \sbra{X_{\infty}^{(k)} = g_{1} \cdots g_{k}} = \mu_{g_1} \frac{\mu_{g_2}}{\sum_{h_2 \in \Next(g_1)} \mu_{h_2}} \cdots \frac{\mu_{g_k}}{\sum_{h_k \in \Next(g_{k-1})} \mu_{h_{k}}}.
\end{equation}
This equation determines the harmonic measure on cylinders and thus determines it completely.

On the other hand as soon as one supposes the random walk is only finitely supported, then this description completely falls apart and much less is known about the harmonic measure. A second contribution of this paper is to show how these results can actually be extended even for general finitely supported random walks. Extending well-known techniques for nearest-neighbour random walks to the coloured setting, the linearization trick then allows to transfer results back to the case of finitely supported convolution walks .

Extensions of harmonic measures and of more general concepts (namely, the Poisson boudary) to the case of coloured random walks have already been considered in \cite{kaimanovich1995poisson}. Here it takes the form of a measure $p^{\infty}_{u}$ on infinite words in the group, $u \in [r]$ being the starting colour. Notice that it does not take colours other than through the starting colour. Yet colours need to be incorporated to recover the Markovian decomposition, which is achieved by considering a matrix version of the parameters $\mu_{g}$, satisfying
\[
p^{\infty}_{u}(\xi_{1} = g) = \sum_{v \in [r]} \mu_{g}(u,v), \: \quad \forall u \in [r].
\]

%Let $(X_{n})_{n \geq 0}$ be a coloured random walk on $G \times [r]$. The definition of $X_{\infty}$ as an infinite word in $G$ extends directly. To incorporate the colour parameter, consider $\tau_{g} := \inf \{ n \geq 0, X_{n} = ( g, \cdot ) \}$ the hitting time of $g$ by the coloured random walk $X_n$. Then for all $k \geq 1$ we consider $u_{k}$ the colour at time $\tau_{X_{\infty}^{(k)}}$ and set for all $g \in S$
%\begin{equation}\label{eq:defmug}
%	\mu_{g}(u,v) := \bP_{(e,u)}[X_{\infty}^{(1)} = g, X_{\tau_{g}}=(g,v)].
%\end{equation}
%In words, $\mu_{g}(u,v)$ is the probability that the random walk, starting from $(e,u)$, visits $g$, with colour $v$ for the first time, and later escapes at infinity in direction $g$. 
%Comparing with the initial case, the process considered in this setting is $(X_{\infty}^{(k)},u_{k})_{k \geq 1}$, which is a coloured Markov chain with increment distribution $\mu$.

Supposing now that $X_{n}$ is nearest-neighbour, the family of parameters $(\mu_g)_{g \in S}$ is uniquely characterized by a family of matrix relations \eqref{eq:traffic}, generalizing the so-called traffic equations of Mairesse \cite{jean2005random}, which arrive as a consequence of stationarity:
\begin{equation}\label{eq:traffic}
	x_{g} = p_{g}  \Delta(x)_{g}  + \sum_{\substack{h,h' \in S \\ hh' = g}} p_{h} x_{h'} + \sum_{h \in \Next(g)} p_{h^{-1}} x_{h}   \Delta(x)_{h} ^{-1} x_{g}, \quad \forall g \in S,
\end{equation}
where $\Delta(x)_{g}$ is the diagonal matrix with entries, for $u \in [r]$,
\begin{equation}\label{eq:defDg}
	\Delta(x)_{g}(u,u) := \sum_{h \in \Next(g)} \sum_{v \in [r]} x_{h}(u,v).
\end{equation} 
In the sequel, $\Delta(\mu)_{g}$ will be simply written $\Delta_{g}$. 

These matrices are also used to extend Equation \eqref{eq:harmonic_nocolour} to the coloured setting. The global description of the harmonic measure can be summed up in the following theorem. We write $\II$ for the vector of $\bR^{r}$ with all coordinates equal to $1$ and $\II_{u}$ for the indicator at $u \in [r]$, that is the vector with zero coordinates except the $u$ entry equal to $1$.

\begin{thm}\label{thm:harmonic}
	For all starting colour $u \in [r]$ and all cylinder $\xi_{1} \cdots \ \xi_{n}$,
	\begin{equation}\label{eq:law_lerw}
		\bP_{u} \sbra{X_{\infty}^{(n)} = \xi_{1} \cdots \xi_{n}, u_{n} = v} =  \II_{u}^{\top} \mu_{\xi_{1}} \Delta_{\xi_1} ^{-1} \mu_{\xi_{2}} \cdots  \Delta_{\xi_{n-1}} ^{-1} \mu_{\xi_{n}} \II_{v},
	\end{equation}
	and
	\begin{equation}\label{eq:harmonic}
		p^{\infty}_{u}(\xi_{1} \cdots \ \xi_{n}) = \II_{u}^{\top} \mu_{\xi_{1}}  \Delta_{\xi_1}^{-1} \mu_{\xi_{2}} \cdots \Delta_{\xi_{n-1}} ^{-1} \mu_{\xi_{n}} \II.
	\end{equation}
	Moreover, the family $\mu = (\mu_{g})_{g \in S}$ is the unique family of matrices with non-negative entries which sum to a stochastic matrix and which are solutions of Equations \eqref{eq:traffic}.
\end{thm}

One motivation for this work was the derivation for finitely supported walks of explicit formulas for various invariant quantities of interest, mainly entropy and drift. By ’explicit’, we mean expressing drift and entropy in terms of the unique solutions of some finite dimensional fixed
point equations. These can be obtained easily from the harmonic measure.

Suppose $\bE \sbra{\abs{X_{1}}} < \infty$, which is in particular true if the walk is finitely supported. The length with respect to the generating set $S$ of $x=(g,u) \in G \times [r]$ is defined as
\begin{equation}\label{eq:|g|}
	|x| := |g| = \min \{k, g = a_{1} \cdots a_{k}, a_{i} \in S \}.
\end{equation}
On the other hand, recall the definition of the entropy:
\[
H(X_{n}) := - \sum_{g} \bP \sbra{X_{n}=g} \log \bP \sbra{X_{n} =g}.
\]
By sub-additivity arguments, eg. Kingman's theorem, one can prove the existence of the following limits:
 \begin{align}
	\gamma &= \lim_{n \rightarrow \infty} \frac{\abs{X_{n}}}{n} \quad \text{a.s. and in $L^{1}$,} \\
	h  &= \lim_{n \rightarrow \infty} - \frac{ \log \cP^{n}((e,u), X_{n})}{n} \quad \text{a.s. and in $L^{1}$}.
\end{align}
$h$, often called the Avez entropy, or entropy rate, or asymptotic entropy, of the random walk, can in fact be interpreted as the drift for the Green pseudo-metric, see \cite{MR2408585}. In this paper it will often simply be referred to as the entropy of the random walk. We refer the reader to Kaimanovich \cite{kaimanovich1983random} for a general reference on the topic in the case of convolution random walks, whereas extended notions of entropy and drift to the setting $G \times [r]$ were considered in \cite{kaimanovich2005munchhausen,MR2762673}.

In the specific case of a linearizing random walk, the initial random walk can be recovered along a sequence of stopping times. Hence the law of large numbers immediately yields the following Abramov formulas. Such formulas are well known, see \cite{MR3199796,MR3667996}.
\begin{coroll}
Let $(X_{n})_{n \geq 0}$ be a finitely supported random walk on $G$ and $(Y_{n})_{n \geq 0}$ a coloured random walk that linearizes $(X_{n})_{n \geq 0}$ in the sense of Definition  \ref{def:linear}. The drift $\tilde \gamma$ and entropy $\tilde h$ of $(Y_{n})_{n \geq 0}$ can be related to the drift $\gamma$ and entropy $h$ of $(X_{n})_{n \geq 0}$ by:
	$\gamma = \bE[\tau_{1}]  \tilde \gamma$ and  $h = \bE[\tau_{1}]  \tilde  h$.
\end{coroll}
We notice also that the expected time $\bE[\tau_1]$ has a simple expression in the two linearization constructions given in Section \ref{section:linearization}, see Remark \ref{rk:Etau} and Remark \ref{rk:Etaurev}.

On the other hand, the following results provide formulas for the drift and entropy coloured random walks. They are direct extensions of known results in the colourless case (corresponding to $r=1$) and the proof will be omitted. The usual scalar product is denoted by $\innerprod{}{}$.

\begin{thm}\label{thm:formula_drift_entropy}
	Let $(Y_{n})_{n \geq 0}$ be a nearest-neighbour quasi-irreducible coloured random walk on $G \times [r]$ defined by a family of matrices $(p_{g})_{g \in S}$. Let $\pi$ be the unique invariant probability measure of the stochastic matrix $P = \sum_{g} p_{g}$. The drift end entropy of $(Y_{n})_{n \geq 0}$ are respectively given by
	\begin{equation*}
		\begin{split}
			\gamma &= \sum_{g \in S} \big\langle {\II} , {\pi p_{g} \big( -\mu_{g^{-1}} + \sum_{h \in \Next(g)} \mu_{h} \big) }\big\rangle.
			\\
			&= \sum_{g \in S}  \sum_{u,v,w \in [r]} \pi(u) p_{g}(u,v) \big( -\mu_{g^{-1}}(v,w) + \sum_{h \in \Next(g)} \mu_{h}(v,w) \big).
		\end{split}
	\end{equation*}
	and
	\begin{equation}\label{eq:entropy_integral}
		h = - \sum_{\substack{g \in G \\ u,v \in [r]}} \pi(u) p_{g}(u,v) \int \log \left(\frac{d g^{-1} p^{\infty}_{u}}{d p^{\infty}_{v}} \right) d p^{\infty}_{v}.
	\end{equation}
\end{thm} 

For nearest-neighbour random walks, the previous integral formula can yield explicit expression for the entropy. However the non-commutativity of $M_r(\bC)$ creates new difficulties. For coloured random walks, adapting the computation naturally leads to the problem of determining the limit of infinite products of random matrices. This in in turn leads to a formula of the entropy in terms of the parameters $(\mu_g)_{g \in S}$ and the law of a family indexed by $[r]$ of random probability measures on $\bR^{r}$ whose law is uniquely characterized by some invariance equation. 

To state our result more precisely, we first introduce the hitting probabilities defined as follows. Given $g \in G, u \in [r]$, let $\tau_{(g,u)} := \inf \{ n : X_{n} = (g,u)\}$ be the hitting time of a pair $(g,u)$, so that $\tau_{g}:= \min_{u} \tau_{(g,u)}$ is the hitting time of $g$. For $g \in S$,  set
\begin{equation}\label{eq:defqg}
	q_{g}(u,v) := \bP[ \tau_{g} < \infty \text{ and } \tau_{g} = \tau_{(g,v)} ].
\end{equation}
As we will check, the matrices $q_g$ and $\mu_g$ satisfy a simple relation \eqref{eq:mu=qD} and the hitting probabilities $(q_g)_{g \in S}$ are characterized by a quadratic equation \eqref{eq:hitting_probabilities} which can be solved quite explicitly for the free group.  In the statement below, $\cP_+$ denotes the open simplex of probability measures $[r]$:
\begin{equation}\label{eq:defcP+}
	\cPr := \left\{ x \in \bR^{r} : \sum_{i \in [r]} x_{i} =1 \text{ and } x_{i} > 0, \: \forall i \in [r]\right\}.
\end{equation}

\begin{thm}\label{thm:formula_entropy_explicit}
	Suppose the matrices $(q_{g})_{g \in S}$ satisfy \eqref{eq:convergence_product} defined in Section \ref{section:computing}. Then the entropy of the coloured random walk is given by
	\begin{equation}
		\begin{split}  h = - \sum_{g  \in S} \sum_{u,v,w\in [r]} \pi(u) p_{g}(u,v) \Bigg[   \mu_{g^{-1}}(v,w) \int \log \left( \frac{\innerprod{\II_{v}}{z}}{\innerprod{q_{g^{-1}}(u,\cdot)}{z}} \right) \: d \nu_{w}(z) \\
			+ \sum_{  h \in \Next(g) } \mu_{h}(v,w) \int \log \left(  \frac{\innerprod{q_{g}(v,\cdot)}{q_{h} z}}{\innerprod{\II_{u}}{q_{h} z}} \right )\: d \nu_{w}(z) \\
			+ \sum_{ h \in S :  gh \in S } \mu_{h}(v,w) \int \log \left( \frac{\innerprod{q_{gh}(v,\cdot)}{z}}{\innerprod{q_{h}(u,\cdot)}{z}} \right) \: d \nu_{w}(z) \Bigg] ,
		\end{split}
	\end{equation}
	where $(\nu_{u})_{u \in [r]}$ is the unique family of probability measures on $\cPr$ satisfying 
	\[
	\int f(z) \: d \nu_{u}(z) = \sum_{\substack{g \in S \\ v \in [r]}}  \int f( q_{g} z) \mu_{g}(u,v) \: d \nu_{v}(z), \quad \forall u \in [r],
	\]
	for all bounded measurable functions $f$ on $\cPr$.
\end{thm}

The technical condition \eqref{eq:convergence_product} is described in Section \ref{section:computing}. Let us simply note that this condition is automatically satisfied for coloured random walks arising as a linearization of a finite range random walk on $G$ (this is the content of Proposition \ref{prop:linlalley} below).
Applying Theorem \ref{thm:formula_entropy_explicit} to a nearest-neighbour walk for which $r= 1, \pi = 1, \nu = \delta_1$, we get back formula (22) in \cite{jean2005random}:
\[
h = - \sum_{g \in S} p_{g} \left( \mu_{g^{-1}} \log \frac{1}{q_{g^{-1}}} + \sum_{h \in \Next(g)} \mu_{h} \log q_{g} + \sum_{h \in S  : \ gh \in S} \mu_{h} \log \frac{q_{gh}}{q_{h}} \right).
\]
As alluded to above, in Theorem \ref{thm:formula_entropy_explicit}, the measure $\nu_g$  on $\cP_+$ will arise as the convergence in direction toward a random rank one projector of the  product of matrices $q_{g_1}q_{g_2}\cdots q_{g_n}$ where $g_1,g_2, \ldots$ in $S$ are the successive letters of $X_\infty$ with law $p^\infty$ conditioned to start with $g_1 = g$. 

For finite range random walks on free products of groups, fine qualitative results on drift and entropy were already available such as  \cite{lalley1993finite,MR2914862} and even on general hyperbolic groups \cite{MR3651925,MR3235351} and references therein. As pointed in \cite{ledrappier2001some,MR3235351}, the computations for drift and entropy for nearest-neighbour random walk on the free group have been known for around 50 years but it was however unexpected that these explicit formulas could be extended to finite range random walks. To our knowledge, Theorem \ref{thm:formula_drift_entropy} and Theorem \ref{thm:formula_entropy_explicit} provide the first formulas for finite range random walks on free products of groups. It would be interesting to 
try to generalize this kind of formulas in other settings such as amalgamated free  products of finite groups or virtually free groups.  

In this paper, we have chosen to focus through the harmonic measure on the drift and entropy but there are other numerical invariants associated to a random walk such as the spectral radius of $\cP$ or the index of exponential decay of the return probability: $\lim_{n \to \infty} \cP^{2n}(e,e)^{1/(2n)}$ (they coincide for reversible walks). Our linearization technique could also be used to study these last two quantities but this is less novel since these quantities can be read on the resolvent operator and previous linearization techniques allow to compute resolvent  of operators of the form \eqref{eq:defcP}, see \cite{MR3718048} for such linearization and \cite{lehner1999computing,MR2216708} for examples of computation of resolvent operators of the form \eqref{eq:deftcP}.

\paragraph{Organization of the paper.}
In Section \ref{section:linearization} we propose a general construction of linearized coloured walks, with some variations, that lead to Theorem \ref{thm:linearization_nonrev}. 
In Section \ref{section:coloured} we define the harmonic measure and prove Theorem \ref{thm:harmonic}. Section \ref{section:computing} contains the arguments leading to Theorem \ref{thm:formula_entropy_explicit}.

\section{Linearization: proof of Theorem \ref{thm:linearization_nonrev}}\label{section:linearization}

Let $(X_{n})_{n \geq 0}$ be a finitely supported random walk on $G$. In this section we construct a nearest-neighbour coloured random walk $(Y_{n})_{n \geq 0}$ from which $(X_{n})_{n \geq 0}$ can be easily recovered. A natural strategy is to try to split transitions into nearest-neighbour steps, so that when projected on $G$, $Y_{n}$ becomes a delayed version of $X_n$. The two constructions that follow are simple refinements of this idea. The first one improves on the naive approach in order to reduce the dimension of the matrices, whereas the second approach adapts the idea to preserve reversibility.

From now on, let  $a_{1}, \ldots, a_{d}$ be the elements of the generating set $S$, and $a_{0} := e$. 
Let $K$ denote the support of the probability measure $p$ and $\ell_K$ be the maximal length of an element of $K$. For each $g \in K$ choose arbitrarily a representative $s(g)$, which can always be assumed to be of length $|g|$. We will generally identify $g$ with its representative and write simply $g=a_{i_{1}} \cdots a_{i_{k}}$ to refer to the fact that $s(g) = a_{i_{1}} \cdots a_{i_{k}}$, for some $i_{1}, \ldots, i_{k} \in [d]$. Similarly, set $p_{i_{1} \cdots i_{k}} := p_{g}$, with in particular $p_{0} := p_{e}$ and $p_{i} := p_{a_{i}}$, whether $a_{i} \in K$ or not. Since we look for a nearest-neighbour coloured random walk, the coloured random walk $(Y_{n})_{n \geq 0}$ is defined by the $r\times r$ matrices $\tp_{i}, i \in [0:d]$.

\paragraph{General Idea}

Suppose $X_{n}$ can jump directly from $e$ to the element $g=a_{i_{1}} \cdots a_{i_{k}} \in K$. The basic idea consists in splitting this one transition into $k$ successive steps: from $e$ to $a_{i_1}$, then from $a_{i_1}$ to $a_{i_{2}}$, etc. The question arising next is the following: once the walk is at $a_{i_1}$ how can it tell whether this step was part of the move towards $g$ or the first step of another transition? This problem is addressed by the introduction of colours. With this extra parameter, it becomes possible to distinguish between transitions that use the same elementary steps. A specific colour, written $1$, will therefore serve as a reference colour: it is the colour telling the walk that it is not in the middle of transition. By opposition, all other colours are only transitional. In particular, the initial random walk is recovered as the trace of the coloured walk along colour $1$. 

The previous ideas are common to all the linearization procedures that we propose. The latter differ in the way colours are introduced and by the transition probability one assigns to elementary steps.

\subsection{Linearization in the non-reversible case}\label{subsec:linsuff}

The naive strategy sketched at the beginning of this section consists in imposing deterministic steps. Taking again the example of a transition $g=a_{i_{1}} \cdots a_{i_{k}} \in K$, we put a probability $p_{g}$ on the transition from $e$ to $a_{i_1}$ and then probability $1$ to all other transitions. This is possible if one allows sufficently many colours, namely $\abs{g} - 1$ colours, exclusive to every generator. Thus the coloured walk $Y_{n}$ goes from $(e,1)$ to $(a_{i_1}, 2)$ with probability $p_{g}$ and then move deterministically from $(a_{i_1}, 2)$ to $(a_{i_2}, 3)$, etc. to eventually arrive at $(g,1)$. The set of colours $\{2, \ldots, \abs{k} \}$ needs to be exclusive to $g$.

It is possible to improve on this idea, noticing that elements of $K$ with common prefixes can share the same colours.

Recall that a representative of minimal length has been fixed for each element of $K$. These representatives are called the representatives of $K$. Given $g,h \in K, g = a_{i_{1}} \cdots a_{i_{k}}, h = a_{j_{1}} \cdots a_{j_{k'}}$, the prefix of $g$ and $h$ is defined as $g \wedge h = a_{i_{1}} \cdots a_{i_{m}}$ where $m := \max \{n \geq 0: i_{1} = j_1, \ldots, i_n  = j_n \}$. $h$ is a strict prefix of $g$ if $g \wedge h = h$ and $\abs{g} > \abs{h}$.
 For $k \in [\ell_K]$ and $i_{1} \ldots, i_{k} \in [d]$, let 
\[
 [i_{1} \cdots i_{k}] := \{ g \in K, g \wedge a_{i_{1}} \cdots a_{i_{k}} = a_{i_{1}} \cdots a_{i_{k}}, \abs{g} \geq k+1 \}
\]
 and
\[
 p_{[i_{1} \cdots i_{k}]} := \sum_{g \in [i_{1} \cdots i_{k}]} p_{g}
\]
% \[
%  q_{i_{1} \cdots i_{k}} := \sum_{j_{k+1}, \ldots j_{L} \in [0,d]} p_{i_{1} \cdots i_{k} j_{k+1} \cdots j_{L}}
% \]
be the cumulative mass under $p$ of all words in $K$ for which $a_{i_{1}} \cdots a_{i_{k}}$ is a strict prefix. Then for $k \in [\ell_K]$, we set 
\[
 q_{i_{1} \cdots i_{k}} = \frac{p_{[i_{1} \cdots i_{k}]}}{p_{[i_{1} \cdots i_{k-1}]}}
\]
where the denumerator is to be taken as equal to one in the case $k=1$. 

The essential idea is the following: we associate a colour $u_{i_1 \cdots i_k}$ to each strict prefix $[i_1 \cdots i_k]$. Suppose now the random walk starts at $(e,1)$ and goes to $a_{i_1}$. With probability $p_{i_1}$ it stays at colour $1$ and with probability $q_{i_1}= p_{[i_1]}$ it takes colours $u_{i_1}$. Suppose this case occurs and the walk continues to $a_{i_1} a_{i_2}$. Then it can come back to colour $1$ with probability $p_{i_{1} i_{2}} / p_{[i_1]}$ or continue to colour $u_{i_{1} i_{2}}$ with probability $q_{i_1 i_2}$. With our choice of transition probabilities, the overall probability to have moved from $(e,1)$ to respectively $(a_{i_1} a_{i_2}, 1)$ and $(a_{i_1} a_{i_2}, u_{i_{1} i_{2}})$ is thus respectively $p_{i_{1} i_{2}}$ and $p_{[i_{1} i_{2}]}$. Iterating this procedure yields the whole generalization, which is formalized in what follows.

Here we will make use of the operator point of view in order to write all the matrices defining $(Y_{n})_{n \geq 0}$ at once. As $M_{r}(\bC) \otimes \cA$ is isomorphic to $M_{r}(\cA)$ for any unital algebra $\cA$, the transition kernel $\tP$ can be written as one matrix whose coefficients are operators on $\ell^{2}(G)$. In this case, the matrix has coefficients which are linear combinations of the multiplication operators. The matrices $p_{i}, i \in [0,d]$ can easily be deduced: to obtain $p_{i}$ it suffices to replace $\rho(a_i)$ by $1$ and the $\rho(a_j), j \neq i$ by $0$.

For all $k \in [\ell_K]$, let 
\[
r_{k} := \mathrm{Card} \ \{(i_{1}, \ \ldots \ , i_{k}) \in S^{k}, [i_{1} \cdots i_{k}] \neq \emptyset \} 
\]
be the number of strict prefixes of elements of $K$ that have length $k$. 
Then, we define the following matrices: $C(k)$ is the $r_{k-1} \times 1$ column matrix 
\[
 C(k) := \left( \sum_{j \in [d]} \frac{p_{i_{1} \cdots i_{k-1} j}}{p_{[i_{1} \cdots i_{k-1}]}} \:\rho(a_j) \right)_{i_{1} \cdots i_{k-1}},
\]
which is indexed by strict prefixes of $K$ of length $k-1$.
Given a representative $i_{1} \cdots i_{k-1}$, define the row matrix 
\[
 L(i_{1}, \ldots, i_{k-1}) := \left( q_{i_{1} \cdots i_{k-1} j} \: \rho(a_{j}) \right)_{j}
\]
indexed by all $j$ such that $i_{1} \cdots i_{k-1} j$ is the strict prefix of a representative of $K$. Then use these row matrices to form the  $r_{k-1} \times r_{k}$ diagonal block matrix $D(k)$, whose $i_{1} \cdots i_{k-1}$ diagonal entry is $L(i_{1}, \cdots, i_{k-1})$:
\[
 D(k) = \begin{pmatrix}
         \ddots &  &  \\
          & L(i_{1}, \ldots, i_{k-1}) &  \\
         & & \ddots
        \end{pmatrix}.
\]
Finally combine all the preceding matrices to construct:
\begin{equation}\label{eq:matrix_nonrev}
 \tilde{\cP} := \begin{pmatrix}
  \sum_{i \in [0:d]} p_{i} \rho(a_{i}) & D(1) & 0 & & & 0\\ 
  C(2) & 0 & D(2) & & &\\
  C(3) & 0 & 0 & D(3) & &\\
  \vdots & & & & \ddots &\\
  C(\ell_K-1) & & & & & D(\ell_K-1) \\
  C(\ell_K) & & & & & 0
 \end{pmatrix}
\end{equation}

% \[
%  C(k) := \left(\begin{matrix} \sum_{j \in [d]} q_{1 \cdots 1j} \rho(a_{j}) \\ \vdots \\ \sum_{j \in [d]} q_{i_{1} \cdots i_{k-1} j} \rho(a_{j}) \\ \vdots \\ \sum_{j \in [d]} q_{d \cdots d j} \rho(j) \end{matrix}\right)
% \]

% \[
%  D(k) := \begin{pmatrix} 
%  L(, \: \ldots, \:  q_{1\ldots 1d}) &  & & &  \\
%   & \ddots & & & \\
%   & & L(q_{i_{1}\cdots i_{k-1} 1}, \: \ldots, \:  q_{i_{1} \ldots i_{k-1} d}) & & \\
%  & & & \ddots & \\
%   & & & & L(q_{d\cdots d1}, \: \ldots, \:  q_{d\ldots dd})
%  \end{pmatrix}.
% \]

% \[
%  L(k) := \left(\begin{matrix}q'_{1 \cdots 11} \rho(a_{1}) & \cdots & q'_{i_{1} \cdots i_{k}} \rho(a_{i_{k}}) & \cdots & q'_{d \cdots dd} \rho(a_{d}) \end{matrix}\right) \qquad C(k) := \left(\begin{matrix}q'_{1 \cdots 11} \rho(a_{1}) \\ \vdots \\ q'_{i_{1} \cdots i_{k}} \rho(a_{i_{k}}) \\ \vdots \\ q'_{d \cdots dd} \rho(a_{d}) \end{matrix}\right)
% \]

% \[
%  D(k) := \begin{pmatrix} 
%  \left(\begin{smallmatrix} \tp_{1 \cdots 1 1} \rho(a_{1}) & \cdots & \tp_{1 \cdots 1 d} \rho(a_{d}) \end{smallmatrix}\right) &  & & &  \\
%   & \ddots & & & \\
%   & & \left(\begin{smallmatrix} \tp_{i_{1} \cdots i_{k-1}1} \rho(a_{1}) & \cdots & \tp_{i_{1} \cdots i_{k-1} d} \rho(a_{d}) \end{smallmatrix}\right) & & \\
%  & & & \ddots & \\
%   & & & & \left(\begin{smallmatrix} \tp_{d \cdots d 1} \rho(a_{1}) & \cdots & \tp_{d \cdots d d} \rho(a_{d}) \end{smallmatrix}\right)
%  \end{pmatrix}.
% \]

\begin{ex}
 The construction will certainly be clearer on a concrete example. Consider a random walk on the group $G$ given by the presentation $\langle a,b \: | \: ab=ba \rangle$, which is in fact isomorphic to $\bZ^{2}$. Suppose the random walk has support
 \[
  K = \{e, a,b, a^{2}, a^{2}b \}.
 \]
To avoid large matrices, we forget about the inverses $a^{-1}, b^{-1}$ but they should not be discarded in a general case. 

Because of the relation $ab=ba$, there can be several ways to write an element of $G$ as words in $a$ and $b$, for instance $a^{2}b = aba$. Therefore we fix a representative for each element of the support. In the present case, group elements can be written uniquely as $a^{k}b^{l}$ so it is natural to choice these words as representatives.

Applying the preceding construction, one eventually obtains the following operator matrix:

\[
 \begin{pmatrix}
  p_{e} \rho(e) + p_{a} \rho(a) + p_{b} \rho(b) & q_{a} \rho(a) & 0\\
  \frac{p_{a^{2}}}{p_{[a]}} \rho(a) & 0 & q_{a^{2}} \rho(a) \\
  \frac{p_{a^{2}b}}{p_{[a^{2}]}} \rho(b) & 0 & 0
 \end{pmatrix} 
= \begin{pmatrix}
	p_{e} \rho(e) + p_{a} \rho(a) + p_{b} \rho(b) & (p_{a^{2}} + p_{a^{2}b}) \rho(a) & 0\\
	\frac{p_{a^{2}}}{p_{a^{2}}+p_{a^{2}b}} \rho(a) & 0 & \frac{p_{a^{2}b}}{p_{a^{2}}+p_{a^{2}b}} \rho(a) \\
	\rho(b) & 0 & 0
\end{pmatrix}. 
\]
\end{ex}

\begin{proof}[Proof of Theorem \ref{thm:linearization_nonrev} (non-reversible case)]

On row $i_{1} \cdots i_{k}$, the sum of entries of the matrix $P= \sum \tp_{i}$ is 
\[
 \sum_{j \in [d]} \frac{p_{i_{1} \cdots i_{k-1} j}}{p_{[i_{1} \cdots i_{k-1}]}} + \sum_{j \in[d]} q_{i_{1} \cdots i_{k-1}j} = \frac{1}{p_{[i_{1} \cdots i_{k-1}]}} \sum_{j \in [d]} (p_{i_{1} \cdots i_{k-1} j} + p_{[i_{1} \cdots i_{k-1} j]}) = 1.
\]
Thus $P$ is stochastic and $\tilde{\cP}$ defines indeed a coloured random walk $(Y_{n})_{n \geq 0}$.

Suppose now $Y_{n} = (g_{n},u_{n})$ is started at colour $1$. Then define $(\tau_{n})_{n \geq 0}$ as the successive return times at the first colour: $\tau_{0} := 0$ and for $n \geq 1$
\[
  \tau_{n} := \inf \{ m > \tau_{n-1} , u_{m}=1\}
 \]
By Markov's property, the random variables $\tau_{n} - \tau_{n-1}$ are iid with the same law as $\tau_{1}$.

For all $g = i_{1} \cdots i_{k}$ in $K$, the probability that $Y_{\tau_{1}} = g$ is 
 \begin{align*}
  \bP \sbra{Y_{\tau_{1}} = g} &= q_{i_{1}} \, q_{i_{1} i_{2}} \, \cdots \ q_{i_{1} \cdots i_{k-1}} \ \frac{p_{i_{1} \cdots i_{k}}}{p_{[i_{1} \cdots i_{k-1}]}} \\ 
  &= p_{[i_{1}]} \frac{p_{[i_{1} i_{2}]}}{p_{[i_{1}]}} \cdots \frac{p_{[i_{1} \cdots i_{k-1}]}}{p_{[i_{1} \cdots i_{k-2}]}} \frac{p_{i_{1} \cdots i_{k}}}{p_{[i_{1} \cdots i_{k-1}]}} \\
  &= p_{i_{1} \cdots i_{k}} = p_{g}.
 \end{align*}
 
 By Markov's property and the fact that the increments are iid one easily deduce that $Y_{\tau_{n}} \stackrel{(d)}{=} X_{n}$ for all $n \geq 0$. Finally, the irreducibility of $X_{n}$ implies the quasi-irreducibility of $Y_{n}$.
\end{proof}

\begin{rk}\label{rk:Etau}
 The expectation of the hitting time $\tau_{1}$ is very simple to compute: $Y_{1}$ can go through each $g \in K$ with probability $p_{g}$, in which case it needs $|g|$ steps to get back to the colour $1$. Hence the expectation of $\tau_{1}$ is just the average length of elements in $K$:
\begin{equation*}\label{eq:Etau1}
  \bE \sbra{\tau_{1}} = p_{e} + \sum_{g \in K} p_{g} |g| = p_{e} + \bE [ |X_1| ].
  \end{equation*}
\end{rk}

\begin{rk}\label{rk:infiniteR}
If $\cP$ is not of finite range, then the construction produces a countable set of colours. The construction could nevertheless be useful if the expected time $\bE [\tau_1] = p_{e} + \bE [|X_1|]$ is finite. 
\end{rk}

\subsection{Linearization in the reversible case}

The previous constructions prevent the coloured random walk to be reversible.
To correct this, we again assign new colours to every element of $K$ with the following variation. We assume that for all $g \in K$, the representative of $g = a_{i_1}\cdots a_{i_n}$ is chosen such that the representative of $g^{-1}$ is $a_{i_n}^{-1}\cdots a_{i_1}^{-1}$. We start with the neutral colour $1$ and, for each pair $(g,g^{-1})$ with $|g| = |g^{-1}| \geq 2$, we add $\abs{g} - 1$ new colours $u_{1}(g), \ldots, u_{\abs{g}-1}(g)$ and we set $u_{k} ( g^{-1} ) = u_{|g|-k}(g)$ for all $1\leq k \leq |g|-1$. Suppose $g$ is written  as $g=a_{i_{1}} \cdots a_{i_{k}}$. For all $h \in G$, the transition probability to go from $(h,1)$ to $(h a_{i_{1}}, u_{1}(g))$ is set to some value $\alpha_{g}$ to be determined, such that $\alpha_{g} = \alpha_{g^{-1}}$ and $\sum_{g \in S} \alpha_{g} =1$. All other transition probabilities on the segment joining $(x,1)$ to $(xg,1)$ are set to $1/2$, see Figure \ref{fig:reversible}.

\begin{figure}
\centering
 \begin{tikzpicture}
 \node[fill, circle, label=left:{$(e,1)$}]  (e) at (0,0) {} ;
 
 \node[fill, shape= circle, label=above:{$a_{i_{1}}$}, label= below:{$u_{1}(g)$}] (1) at (2,0) {};
 \node[fill, shape= circle, label=above:{$a_{i_{1}} a_{i_{2}}$}, label= below:{$u_{2}(g)$}] (2) at (4,0) {};
 \node[fill, shape= circle, label=above:{$a_{i_{1}} \cdots a_{i_{k-2}}$}, label= below:{$u_{k-2}(g)$}] (n-2) at (6,0) {};
 \node[fill, shape= circle, label=above:{$a_{i_{1}} \cdots a_{i_{k-1}}$}, label= below:{$u_{k-1}(g)$}] (n-1) at (8,0) {};
 \node[fill, shape= circle, label= right:{$(g,1) \in K \times \{1\}$}] (g) at (10,0) {};
 
%  \node[fill, shape= circle] (h') at (12,0) {};
% \node[fill, shape= circle] (h'1) at (11.414,1.414) {};
%  
 \node[coordinate] (h) at (110:1.5) {} ;
 \node[coordinate,label=above left:{$K \times \{1\}$}] (h2) at (135:1.5) {};
 \node[coordinate] (h3) at (-135:1.5) {} ;
 \draw[loosely dashed] (e) edge (h) edge (h2) edge (h3) ;
%  \draw[loosely dotted] (h2) to[bend right=45] (h3) ;
 
 \draw[->, >=stealth] (e) to[bend left] node[midway,above]{$\alpha_{g}$} (1) ;
 \draw[->,>=stealth](1) to[bend left] node[midway,below]{$1/2$} (e) ;
 \draw[<->, >=stealth] (1) to node[midway,below]{$1/2$} (2) ;
 \draw[<->, >=stealth, dashed] (2) to (n-2) ;
 \draw[<->, >=stealth] (n-2) to node[midway,below]{$1/2$} (n-1) ;
 \draw[->, >=stealth] (n-1) to[bend left] node[midway,above]{$1/2$} (g) ;
 \draw[->,>=stealth](g) to[bend left] node[midway,below]{$\alpha_{g}$} (n-1) ;
 
 \draw[loosely dashed] (g) --++ (45:1.5) ;
 \draw[loosely dashed] (g) --++ (80:1.5) ;
 \draw[loosely dashed] (g) --++ (-45:1.5) ;
\end{tikzpicture}
\caption{The linearizing reversible random walk}
\label{fig:reversible}
\end{figure}
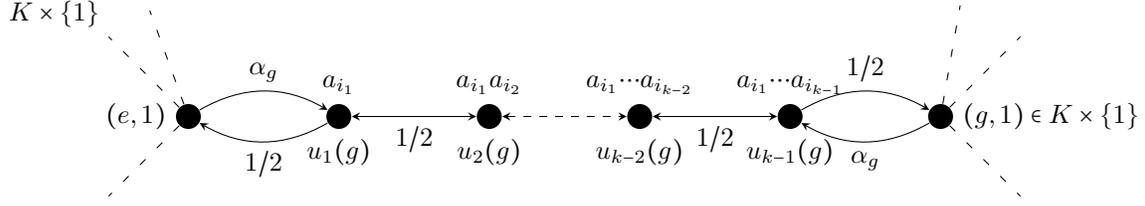

Thanks to the constraint $\sum_{g \in G} \alpha_g = 1$, $ \tilde{\cP} $ is indeed a Markov transition kernel. However, the coloured walk can now make steps that would not have been possible in the naive approach. For instance, it can go from $(e,1)$ to some $(g,u)$ and then come back to $(e,1)$, even if $p_{e} = 0$ for the initial walk. Therefore, the stopping times cannot be taken directly as the return times at colour $1$. Instead define, if $Y_n = (g_n,u_n)$ has transition kernel $\tcP$,
\[
  \tau_{n} := \left\{ \begin{array}{l l}
               \inf \{ m \geq \tau_{n-1} :  u_{m}=1, g_{m} \neq g_{\tau_{n-1}} \} & \text{if $g_{\tau_{n-1}+1} \neq g_{\tau_{n-1}}$} \\
               \tau_{n-1}+1 & \text{if $g_{\tau_{n-1}+1} = g_{\tau_{n-1}}.$}
              \end{array} \right.
 \]
 
Theorem \ref{thm:linearization_nonrev} in the reversible case is now contained in the following lemma.

\begin{lem}
 Suppose $X_{n}$ is a finitely supported random walk on $G$ defined by a probability vector $p$ and started at $e$. Setting for each $g \in K$
  \begin{equation}\label{eq:proba_reversible}
  \alpha_{g} := \left\{ \begin{array}{l l} 
                         \frac{(1-p_{e})|g| p_{g}}{\sum_{h \in K} |h| p_{h}} & \text{if $g \neq e$} \\
  p_{e} & \text{if $g = e$}
                        \end{array}\right.
 \end{equation}
the preceding construction yields an operator $\tilde{\cP}$ which defines a reversible  coloured random walk $(Y_{n})_{n \geq 0}$ satisfying $Y_{\tau_{n}} \stackrel{(d)}{=} X_{n}$ for all $n \geq 0$ if started at $(e,1)$.
\end{lem}

\begin{proof}
As already pointed, the construction defines a coloured random walk as long as $\sum_{g \in S} \alpha_{g} =1$, which is true for $\alpha_{g}$ defined as in \eqref{eq:proba_reversible}. Furthermore such a coloured walk is reversible: for all $(g,u) \in G \times [r]$, letting
\[
 \mu(g,u) = \left\{ \begin{array}{l l}
             1 & \text{if $u=1$} \\
             2 \, \alpha_{h} & \text{if $u=u_{h}(k)$ for some $k < \abs{h}$}.
            \end{array}\right.
\]
defines a reversible measure for the random walk, as it can be checked directly.

Consider now the coloured walk $(Y_{n})_{n \geq 0}$ defined by \eqref{eq:proba_reversible}, started at $Y_{0}=(e,1)$. Suppose first that $p_{e}=0$, so that $\tau_{1}$ is the hitting time of the set $(K \smallsetminus \{e\}) \times \{1\}$. 
 
At step $1$ the coloured walk necessarily enters the segment joining $(e,1)$ to $(h,1)$ for some $h \in K$.  If the walk escapes this segment at $h$, then $Y_{\tau_{1}} = h$. Otherwise, it necessarily goes back to the starting state $(e,1)$. Now on each of these segments, the coloured walk is simply performing a simple random walk so the escape probabilities are given by the standard gambler's ruin problem. Namely the simple random walk on $[0:n]$ reaches $n$ before it gets to $0$ with probability $k/n$ when started at $k \in [0:n]$. Therefore by Markov's property
 \[
  \bP_{(e,1)}\sbra{Y_{\tau_{1}} = g}= \alpha_{g}/\abs{g} + \sum_{h \in K} \alpha_{h}(1-1/\abs{h}) \bP_{(e,1)}\sbra{Y_{\tau_{1}}=g}.
 \]
We deduce that 
\[
 \bP_{(e,1)}\sbra{Y_{\tau_{1}} = g} = \frac{\alpha_{g} / \abs{g}}{\sum_{h \in K} \alpha_{h} / \abs{h}} = p_{g}.
\]
 
For the general case $p_{e} \neq 0$, $Y_{\tau_{1}} = (e,1)$ if and only if $Y_{1} = (e,1)$ which occurs with probability $\alpha_{e} = p_{e}$. For $g \neq 0$, consider the random walk conditionned to move at every step, written $(Y'_{n})_{n \geq 0}$. If $q(x,y)$ is the transition probability for $Y_{n}$ between two states $x$ and $y$, then the transition probability for $Y'_{n}$ is $0$ if $y=x$ and $q(x,y)/(1-q(x,x))$ otherwise. The previous argument apply to this walk, so
\[
 \bP_{(e,1)}\sbra{Y'_{\tau'_{1}} = g} = \frac{\alpha_{g} / ((1-p_{e})\abs{g})}{\sum_{h \in K \smallsetminus \{e\}} \alpha_{h} / ((1-p_{e})\abs{h})} = \frac{p_{g}}{1-p_{e}},
\]
with $\tau'_{1}$ is the obvious extension of $\tau_{1}$ to $Y'_{n}$.
Coming back to $Y_{\tau_{1}}$: to reach $g \neq e$ it it is necessary that $Y_{1} \neq e$, which occurs with probability $(1-p_{e})$. Conditional on that event it no longer matters whether the random walk comes back at $e$ and possibly stays there, so one can reason with $Y'_{n}$ instead. Hence
\[
 \bP_{(e,1)}\sbra{Y_{\tau_{1}} = g} = (1-p_{e}) \bP_{(e,1)}\sbra{Y'_{\tau'_{1}} = g} = p_{g}.
\]
The conclusion follows.
\end{proof}

\begin{rk}\label{rk:Etaurev}
 The expected time $\bE\sbra{\tau_{1}}$, although not as simple as in the non-reversible case, can nonetheless be computed quite easily using for instance the electric network paradigm. Using \cite[Prop 2.20]{lyons2017probability}, we found in the case $p_{e} = 0$
  \begin{equation}
   \bE\sbra{\tau_{1}} = \sum_{g \in K} p_{g} \abs{g}^{2} = \bE \abs{X_{1}}^{2}.
  \end{equation}
\end{rk}

\begin{rk}\label{rk:bdr}
In the naive approach, we have that the total number of colours is 
$$
r = 1 + \sum_{g \in K} (|g| -1).
$$
In the reversible construction, there is a factor $1/2$ on front of the sum (because we use the same colours for $g$ and $g^{-1}$). In the improved construction of Subsection \ref{subsec:linsuff}, this is an upper bound, the actual value is $r = 1 + \mathrm{Card}(K') - \mathrm{Card}(K)$ where $K'$ is the set of words which are a suffix of some element in $K$ (in the chosen representatives). In a concrete application, it is often possible to design  linearization procedures which are more economical in terms of number of colours.
\end{rk}

\section{Coloured random walks on free products of groups}\label{section:coloured}

In this section we define and study the harmonic measure for coloured random walks on such groups, proving Theorem \ref{thm:harmonic}. 
Explicit descriptions of the harmonic measure on free groups were obtained in \cite{kaimanovich1983random,derriennic1980quelques,ledrappier2001some}. Similar approaches can be carried out for free products of finite groups, of monoids, of finite alphabets \cite{MR2378433, gilch2011asymptotic,mairesse2007random}. 

The main ingredient in the study of random walks on groups is the invariance by translation of the transition probabilities. This $G$-invariance of transition probabilities still holds for coloured random walks, allowing many arguments to extend to this setting. Indeed, once a good way of taking colours into account is found, one can essentially replace scalar parameters (such as $p_{g}, \mu_{g}$, etc.) by matrices. The argumentation for convolution random walks can then be extended in a straightworward way as long as the computations are consistent with matrix products, which is the case most of the time. For this reason, we will often sketch or omit the proofs when they are too similar to already known arguments. We note however that detailed proofs are available in the first version of this paper \cite{bordenave2020markovian}.

% Given two groups $G_{1}, G_{2}$, the free product $G_{1} \ast G_{2}$ of $G_{1}$ and $G_{2}$ is basically the group whose elements are words $g_{1} g_{2} \cdots \ g_{k}$ where no two successive elements come from the same group. The group law is the concatenation of words, with possible simpification arising from srelations in $G_{1}, G_{2}$. Iterate this construction to define free products of multiple groups. The free group on $d$ generators in particular is the free product of $d$ copies of $\bZ$.

The case of free groups and free products of finite groups is very similar in nature and they can be handled together as done in \cite{jean2005random}.
 Let $G_{1}, \ldots, G_{m}$ be finite groups and consider the group $G = \Fd \ast G_{1} \ast \cdots \ast G_{m}$, with the set of generators \eqref{eq:generators}. 
 Recall the definition of the map $\Next$ in \eqref{eq:next}. Every element $g \in G$ writes uniquely as a word $g=g_{1} \cdots g_{n}$ with $n = \abs{g}$ and $g_{i} \in S, g_{i+1} \in \Next(g_{i})$ for all $i$. Such words will be called reduced.  

\subsection{The harmonic measure on the boundary}

Consider a coloured random walk $(X_{n})_{n \geq 0}$ on $G \times [r]$ defined by a family $(p_g)$ of matrices. We assume that the walk is quasi-irreducible and nearest-neighbour. 
We assume furthermore that $G$ is not isomorphic to $\bZ$ or $\bZ/2\bZ \ast \bZ/2\bZ$. This implies the coloured walk is transient, as shown by Proposition \ref{prop:ex+uni_stationary} below. %Transience justifies the following approach to compute the drift and entropy of $(X_{n})_{n \geq 0}$ with respect to the set of generators $S$.

Define the boundary $\partial G$ as 
\[
 \partial G := \left\{ \xi_{0}\xi_{1} \cdots \xi_{n} \cdots , \forall i \geq 0, \xi_{i} \in G, \xi_{i+1} \in \Next(\xi_{i}) \right\}.
\]
The multiplication action by $G$ on itself can be extended to $\partial G$: given $g \in G$, $\xi = \xi_{0} \xi_{1} \cdots \in \partial G$, define
\begin{equation}\label{eq:defgxi}
 g \xi := \left\{ \begin{array}{l l}
   g \xi_{0} \xi_{1} \cdots & \text{if $\xi_{0} \in \Next(g)$} \\                     
   (g \xi_{0}) \xi_{1} \cdots & \text{if $g \xi_{0} \in S$} \\              
   \xi_{1} \xi_{2} \cdots & \text{if $g = \xi_{1}^{-1}$.} 
 \end{array}\right. 
\end{equation}
The boundary $\partial G$, which can be seen as a subset of $G^{\bN}$, is equipped with the product topology and the corresponding $\sigma$-algebra. Given a measure $\nu$ on $\partial G$, let $g \cdot \nu$ be the image measure of $\nu$ under the mutiplication by $g$, that is the measure defined by the fact that 
\[
 \int f(\xi) \: d (g \cdot \nu)(\xi) := \int f( g \xi) \: d \nu(\xi),
\]
for all bounded measurable function $f$ on $\partial G$. 

\begin{Def}
 Consider a family $\nu = (\nu_{u})_{u \in [r]}$ of probability measures on $\partial G$ indexed by colours. $\nu$ is said to be stationary if for all $u \in [r]$,
\begin{equation}\label{eq:defstat}
 \nu_{u} = \sum_{g \in G, v \in [r]} p_{g}(u,v) \: g \cdot \nu_{v}.
\end{equation}
\end{Def}

From now on, all measures considered on $\partial G$ will be implicitely indexed by colours.
The following result extends Theorem 1.12 in \cite{ledrappier2001some} to coloured random walks. The proof is exactly the same and will thus be omitted.

\begin{prop}\label{prop:ex+uni_stationary}
 There exists a random variable $X_{\infty} \in \partial G$ such that $X_{n}$ converges a.s.\,to $X_{\infty}$. The law of $X_{\infty}$ is called the harmonic measure and is the unique stationary measure on the boundary $\partial G$. It will be denoted $(p^{\infty}_{u})_{u \in [r]}$ where the index $u$ is to be interpreted as the starting colour of $(X_{n})_{n \geq 0}$.
\end{prop}

\subsection{Markovian structure of the harmonic measure}

A measure on $\partial G$ is uniquely determined by its mass on cylinders. Given a measure $\nu$, we will write $\nu(\xi_{1} \cdots \xi_{n})$ for the mass of the cylinder containing all infinite words which start with the prefix $\xi_{1} \cdots \xi_{n}$.

In the colourless case, the tree structure of the group $G$ implies the harmonic measure is Markovian. It can be computed entirely from the solutions of a set of equations derived from the stationarity of the harmonic measure with the following interpretation.

For all $k \geq 0$, let $X_{\infty}^{(k)}$ be the restriction of $X_{\infty}$ to the first $k$-th letters. 
Thus for all $k \geq 1$ the mass under the harmonic measure of cylinders of size $k$ is given by the law of $X_{\infty}^{(k)}$. 
On the other hand $(X_{\infty}^{(k)})_{k \geq 1}$ is a non-backtracking walk which one can interpret as the loop-erased random walk formed from $(X_{n})_{n \geq 0}$. In the standard colourless setting, the tree structure of the group makes the loop-erased random walk a Markov chain whose transition probabilities can be computed easily.

In the coloured setting, one can expect to have similar properties but the loop-erased random walk of the process is no longer a Markov chain so this is not the right process to consider.
Instead let $\tau_{g} := \inf \{ n \geq 0, X_{n} = ( g, \cdot ) \}$ be the hitting time of $g$ by the random walk $X_n$ and $u_{k}$ the colour at time $\tau_{X_{\infty}^{(k)}}$. 
%Thus $u_{k}$ is the first colour seen by the random walk $X_{n}$ when it visits the same element as the $k$-th step of $X_{\infty}$, but it is not necessary the colour visited at the step $X_{n}$ where the $k$-th letter of walk starts coinciding with $X_{\infty}^{(k)}$ forever.
Given $g \in S$, set 
\begin{equation}\label{eq:defmug}
 \mu_{g}(u,v) := \bP_{(e,u)}[X_{\infty}^{(1)} = g, X_{\tau_{g}}=(g,v)].
\end{equation}
In words, $\mu_{g}(u,v)$ is the probability that the random walk, starting from $(e,u)$, visits $g$, with colour $v$ for the first time, and later escapes at infinity in direction $g$. 
The process considered which is the equivalent of the loop erased random walk in the coloured setting is $(X_{\infty}^{(k)},u_{k})_{k \geq 1}$, which is a coloured Markov chain with increment distribution $\mu$.

\begin{proof}[Proof of Theorem \ref{thm:harmonic}]
 Most of the proof is similar to the case of colourless walks, so we will only emphasize what is different. Details can be found in \cite{ledrappier2001some, jean2005random, mairesse2007random}.
 
 Recall Definition \eqref{eq:defDg} and that $\Delta_{g} := \Delta(\mu)_{g}$. Given the definition of the matrices $\mu_{g}$, Equation \eqref{eq:law_lerw} is an easy consequence of Markov's property and the $G$-invariance of the transition probabilites of the random walk. 
 
 Let us now prove now that the matrices $\mu_{g}$ are characterized by Equation \eqref{eq:traffic}. 
 Consider any family $(\nu_{g})_{g \in S}$ of non-negative matrices, solutions of \eqref{eq:traffic} and such that $\sum_{g \in S} \nu_{g}$ is a stochastic matrix. Then for all $u \in [r]$ define the measure $\nu^{\infty}_{u}$ on $\partial G$ by setting for each cylinder $\xi_{1} \cdots \ \xi_{n}$,
 \[
  \nu^{\infty}_{u} = \II_{u} \nu_{\xi_{1}}   \Delta(\nu)_{\xi_1} ^{-1}  \nu_{\xi_{2}} \cdots  \Delta(\nu)_{\xi_{n-1}}  ^{-1} \nu_{\xi_{n}} \II
 \]
 The fact that $\sum_{g \in S} \nu_{g}$ is a stochastic matrix proves that $\nu^{\infty}_{u}$ is indeed a probability measure on $\partial G$. 

 Mimicking the proof of the case of colourless walks, one can deduce that $\nu^{\infty}$ is stationary, which by uniqueness of stationary measures yields that $\nu^{\infty}_{u} = p^{\infty}_{u}$ for all $u \in [r]$. 
  
 At this stage we only proved that for all $g \in S, u \in [r]$, $p^{\infty}_{u}(g) = \sum_{v \in [r]} \mu_{g}(u,v) = \sum_{v \in [r]} \nu_{g}(u,v)$. To deduce $\mu_{g} = \nu_{g}$, we observe that the previous equality shows $\Delta_{g} = \Delta(\nu)_{g}$ for all $g \in S$. Consequently the matrices $\mu_{g}$ and $\nu_{g}$ can be seen as the fixed points of a quadradic map $f: M=(M_{g})_{g \in S} \mapsto (f(M)_{g})_{g \in S}$ from the set of non-negative matrices to itself, where
 \begin{equation*}\label{eq:deffM}
  f(M)_{g} :=  A_{g} + \sum_{\substack{h,h' \in S \\ hh' = g}} p_{h} M_{h'} + \sum_{h \in \Next(g)} p_{h^{-1}} M_{h} B_{g} M_{g},
 \end{equation*}
 for all $g \in S$, and $A_{g}, B_{g}$ are non-negative matrices that do not depend on $M$. 
 
 By the fact that $f$ is quadratic, one can now use monotonicity arguments as in the proof of \cite[Lem. 4.7]{jean2005random} to deduce that $\mu_g = \nu_g$.
\end{proof}

\paragraph{Hitting probabilities.}

Let us discuss here another way to compute the harmonic measure through hitting probabilities.  It is an easy extension to the coloured case of equations that can be found for classical nearest-neighbour random walks in \cite{ledrappier2001some, jean2005random, mairesse2007random} but here the equations become matrix equations.

Recall Definition \eqref{eq:defqg}. The motivation for introducing these new matrices is that they can be easily related to the matrices $\mu_{g}$ by
\begin{equation}\label{eq:mu=qD}
	\mu_{g}(u,v) = q_{g}(u,v)  \Delta_{g}(v,v), \: \quad \forall g \in S.
\end{equation}

This gives yet another way to write equations \eqref{eq:law_lerw} and \eqref{eq:harmonic}:
\begin{eqnarray}
	\bP_{u} \sbra{X_{\infty}^{(k)} = \xi_{1} \cdots \xi_{k}, u_{k} = v} & = &\left( q_{\xi_{1}} \cdots q_{\xi{k-1}} \mu_{\xi_{k}} \right)(u,v), \label{eq:law_lerw2} \\
	p_{u}^{\infty}(\xi_{1} \cdots \xi_{k}) & = & \sum_{v} \left( q_{\xi_{1}} \cdots q_{\xi{k-1}} \mu_{\xi_{k}} \right)(u,v) \label{eq:harmonic2}.
\end{eqnarray}

Furthermore, using \eqref{eq:mu=qD}, the traffic equation \eqref{eq:defmug} can be rewritten as
\begin{equation}\label{eq:trafficq}
	\mu_{g} = p_{g}  \Delta_{g}  + \sum_{\substack{h,h' \in S \\ hh' = g}} p_{h} \mu_{h'} + \sum_{h \in \Next(g)} p_{h^{-1}} q_{h} \mu_{g}, \quad \forall g \in S.
\end{equation}

Observe that $y_g := \Delta_g \II = \mu_g \II$ and $y_g \in \bR^r$ is equal to the diagonal of the diagonal matrix  $\Delta_{g}$. By construction, we have that $\sum_g y_g = \II$. If we evaluate the matrix equation \eqref{eq:trafficq} on the vector $\II$, we obtain a linear equation for $y = (y_g)_{g\in S}$, seen as a  vector with coordinates in $\bR^r$, of the form $T y = 0$ where $T$ is matrix on $S \times S$ with matrix-valued coefficients in $M_r(\bR)$. Finally, once, $\Delta_g$ is known, Equation \eqref{eq:trafficq} becomes linear in $\mu$ seen as a vector of matrices. Therefore the matrices $\mu_{g}$ can be completely recovered from the $q_{g}$. 

On the other hand, the latter satsfy a simpler quadratic equation: conditioning on $X_{1}$ and applying Markov property, one obtains
\begin{equation}\label{eq:hitting_probabilities}
 x_{g} = p_{g} + \sum_{\substack{h,h' \in S \\ h h' = g}} p_{h}x_{h'} + \sum_{h \in \Next(g)} p_{h^{-1}} x_{h} x_{g} .
\end{equation}
% Indeed $q_{g}(u,v)$ is the probablity, starting from $(e,u)$, that the walk ever visits $g$ and that first colour attained is $v$. Equation \eqref{eq:hitting_probabilities} is obtained by considering the first step of the random walk: to visit $g$, either the walk goes to $g$ already at the first step, or it goes to $h \neq g$. In this case either one multiplication by an element $h' \in S$ is sufficient to get to $g$, this happens for instance if $g,h,h'$ are in a same group $G_{i}$,  or the random walk needs to backtrack. These three different cases give three different terms in the right-hand side above. 
Arguing as in the proof of Theorem \ref{thm:harmonic}, one can prove this equation characterizes the $q_g$ as follows:

\begin{lem}\label{le:hitprob} The family $q = (q_{g})_{g \in S}$ is the unique solution to Equation \eqref{eq:hitting_probabilities} among family of matrices $(m_g)_{g\in S}$ such that $\mu_{g}(u,v) \leq m_{g}(u,v) \leq 1$  for all $u,v \in [r]$.
\end{lem}

%\begin{proof}
% One can argue as in the proof of lemma \ref{lem:uniqueness_mu} or the proof of \cite[Lemma 4.7]{jean2005random}: solutions of Equation \eqref{eq:hitting_probabilities} can be seen as fixed points of quadratic maps that are non-increasing with respect to the coordinate-wise ordering, which cannot have two fixed points.
%\end{proof}

\subsection{Solving the traffic equations}
\label{subsec:trafficfree}

We now propose methods to solve the equations of the previous sections. For a general free product of groups, we show how solutions can be computed numerically. When $G= \Fd \ast \bZ / 2 \bZ \ast \cdots \ast  \bZ / 2 \bZ $, these equations can be solved explicitely in the case of classical nearest-neighbour random walks. For coloured random walks, complications arise due to matrix equations, so we can give explicit solutions up to the issue of determining a matrix square root. Thanks to the linearization trick, this problem is thus the only remaining obstacle preventing from obtaining an explicit description for finite range random walks. 

\paragraph{General case.}
We have seen in the previous section that the harmonic measure is fully described by the family of matrices $\mu = (\mu_g)_{g \in S}$ on $[r]$ defined by \eqref{eq:defmug}. By Theorem \ref{thm:harmonic}, these matrices are uniquely characterized by the traffic equation \eqref{eq:traffic} so it is possible to evaluate numerically $\mu$ by iterating the map defining the traffic equations. 

The hitting probabilities $q = (q_g)_{g \in S}$ satisfy a simpler quadratic equation \eqref{eq:hitting_probabilities} than the traffic equation of $\mu$. As explained in the proof of Lemma \ref{le:hitprob}, it is possible to evaluate numerically from above and from below $q$ by iterating the map $f$ defined in \eqref{eq:hitting_probabilities} such that $q = f(q)$.

\paragraph{Case of $G= \Fd \ast \bZ / 2 \bZ \ast \cdots \ast  \bZ / 2 \bZ $.} In some special cases, one can find additional relations allowing to write the matrix $\mu_{g}$ as an explicit function of the matrices $q_{g}$. This case occurs for instance when $G= \Fd \ast \bZ / 2 \bZ \ast \cdots \ast  \bZ / 2 \bZ $. Given $g \in S$, let $d_{g}$ be the diagonal matrix with diagonal entries
\[
 d_{g}(u,u):=  \sum_{w} q_{g}(u,w).
\]
In words, $d_{g}(u,u)$ is the probability to ever reach $g$, starting from the pair $(e,u)$. 

\begin{prop}\label{prop:muqfree}
 When $G= \Fd \ast \bZ / 2 \bZ \ast \cdots \ast  \bZ / 2 \bZ $, for all $g \in S$ the matrix $(I-q_{g^{-1}} q_{g})$ is invertible and 
 \begin{equation}
  \mu_{g}(u,v) = q_{g}(u,v) \sum_{w \in [r]} (I-q_{g^{-1}} q_{g})^{-1} (I-d_{g^{-1}})(v,w) \quad \forall g \in S, u,v \in [r].
 \end{equation}
\end{prop}

\begin{proof}
Suppose $G= \Fd \ast \bZ / 2 \bZ \ast \cdots \ast  \bZ / 2 \bZ $. Then for all $g \in S, u \in [r]$,
 \[
  \mu_{g}(u,v) = q_{g}(u,v) \left( 1-\sum_{w \in [r]} q_{g^{-1}}(v,w) + \sum_{w,z \in [r]} q_{g^{-1}}(v,w) \mu_{g}(w,z) \right).
 \]
In particular $q_{g}(u,v) = 0$ implies $\mu_{g}(u,v) = 0$. Otherwise, rewrites this equation as
 \[
  \frac{\mu_{g}(u,v)}{q_{g}(u,v)} = 1 - d_{g^{-1}}(v,v) + \sum_{w,z \in [r]} q_{g^{-1}}(v,w) q_{g}(w,z) \frac{\mu_{g}(w,z)}{q_{g}(w,z)}.
 \]
 The right-hand being independant of $u$, so is the left-hand side. Consequently let $x_{v} := \mu_{g}(u,v) / q_{g}(u,v)$. Then one can again rewrite the above equation as the matrix linear equation 
 \[
  x = (I-d_{g^{-1}}) \II + q_{g^{-1}} q_{g} x.
 \]
 Provided $I - q_{g^{-1}} q_{g}$ is invertible, solving this equation yields the desired expression for $\mu_{g}(u,v)$.
 
 Therefore we are left to prove that $I-q_{g^{-1}} q_{g}$ is invertible. This can be justified by the fact that $\sum_{v \in [r]} \sum_{n \geq 0} (q_{g^{-1}} q_{g})^{n}(u,v)$ is the average number of times the walk goes to $g^{-1}$ and comes back to $e$ when starting at colour $u$. By transience of the walk, this number must be finite. Hence the sum $\sum_{n \geq 0} (q_{g^{-1}} q_{g})^{n}$ is convergent and the inverse of this matrix is precisely $I - q_{g^{-1}} q_{g}$. 
\end{proof}

We now further study the matrix equation \eqref{eq:defqg} satisfied by the hitting probabilities $(q_g)_{g \in S}$. Supposing $G= \Fd \ast \bZ / 2 \bZ \ast \cdots \ast  \bZ / 2 \bZ $, for all $g \in S$, Equation \eqref{eq:hitting_probabilities} becomes: 
\begin{equation}\label{eq:qfree}
q_{g} = p_{g} + \sum_{h \ne g^{-1}} p_{h^{-1}} q_{h} q_{g}.
\end{equation}
In the colourless case, it is possible to reduce Equation \eqref{eq:qfree} to a scalar equation, see for example \cite{ledrappier2001some}. We extend this computation to the coloured case. We define the matrix in $M_r(\bR)$, 
\begin{equation}\label{eq:defsigma}
z = I - \sum_{g \in S}  p_{g^{-1}} q_{g}.
\end{equation}
We now express $q_g$ as a function of $(z,p_g,p_{g^{-1}})$ and find a closed equation satisfied by $z$. For simplicity, we assume that for all $g \in G$,  the matrix $p_g$ is invertible. If this is not the case, a similar argument holds but one should be careful with pseudo-inverses.

We define the matrices in $M_{2r} (\bR)$: 
$$
P_g = \begin{pmatrix}
0 & p_g \\
p_{g^{-1}} & 0
\end{pmatrix} \; , \quad Z = \begin{pmatrix}
z & 0 \\
0 & z 
\end{pmatrix}  \quad \hbox{ and } \quad  Q_g = \begin{pmatrix}
0 & q_{g} \\
  q_{g^{-1}}  & 0 
\end{pmatrix}.
$$
From our assumption, $P_g$ is invertible and 
\[
P_{g}^{-1} = \begin{pmatrix}
	0 & (p_{g^{-1}})^{-1} \\
	p_{g}^{-1} & 0
\end{pmatrix}.
\]  
Applying \eqref{eq:qfree} to $g$ and $g^{-1}$, we find that, for all $g \in S$, $zq_g  = p_g - p_g q_{g^{-1}} q_g$ and thus
$$
Z Q_g  = P_g - P_g Q_g^2.
$$

We set 
$
Z_g = P_g^{-1} Z 
$. The above equation rewrites
\begin{equation}
\label{eq:tZQ}
(Z_{g} + Q_{g}) Q_{g} = I.
\end{equation}
In particular $Z_{g} = Q_{g}^{-1} - Q_{g}$ so $Z_{g}$ and $Q_{g}$ commute. We may thus solve the quadratic equation $Q_g^2 + Z_{g} Q_{g} - I = 0$ with $Q_g$ as unknown as in the scalar case. Completing the square yields $$(Q_g + Z_g /2)^2 = I + Z_{g}^{2} / 4.$$ 
Therefore, for some proper choice of the matrix square root function, we get
\begin{equation}\label{eq:tQS}
Q_g = \frac 1 2 \left( \sqrt[*]{ 4 I + (P_g^{-1} Z)^2 }  - P_g^{-1} Z \right),
\end{equation}
where $\sqrt[*]{\cdot}$ is a notation to stress that the choice of the square root is unknown.  First, as $Z_{g}$ and $Q_{g}$ commute, the eigenvalues completely determine the square root. Also, since $Q_g$ is a block antidiagonal matrix, for every of its eigenvalue $\lambda$, $- \lambda$ must also be an eigenvalue with the same multiplicity (algebraic and geometric). In particular, we are left with at most $2^r$ choices for the square root to pick in \eqref{eq:tQS} (a choice of sign for each eigenvalue pair $(\lambda,-\lambda)$). 

There is one useful property  to further determine the square root. Consider 
$$
R_g = Q_g P_g^{-1}  = \begin{pmatrix}
q_g p_{g}^{-1} & 0 \\
0 & q_{g^{-1}} p_{g^{-1}}^{-1}
\end{pmatrix}.
$$
We observe from \eqref{eq:qfree} that $q_g p_g^{-1} = (I - \sum_{h \ne g^{-1}} p_{h^{-1}} q_h)^{-1}$. The matrix $\sum_{h \ne g^{-1}} p_{h^{-1}} q_h$ is sub-stochastic: it has non-negative entries and the sum over each row is less or equal than one (from a starting colour, it corresponds to the probability that the coloured walk killed when visiting $g$ comes back to $e$ after some time). In particular, the matrix  $\sum_{h \ne g^{-1}} p_{h^{-1}} q_h$ has spectral radius less than one. Thus all eigenvalues of $q_g p_g^{-1}$ and $R_g$ have positive real part.  Using \eqref{eq:defsigma}, the same argument shows that $z$ and $Z$ have all their eigenvalues with positive real parts.

Finally, from Equation \eqref{eq:defsigma}, we have $Z = I - \sum_g P_g Q_g = I - \sum_g P_g R_g P_g$. We deduce that  
\begin{equation}\label{eq:Spf}
Z  =   I - \frac 1 2 \sum_{g \in S} \left( P_g  \sqrt[*]{ 4 I + ( P_g^{-1} Z ) ^2 }  - Z\right). 
\end{equation}

Up to this issue of square root, we thus have found a fixed point equation satisfied by $z$ (in Equation \eqref{eq:Spf}) and expressed $q_g$ as a function of $(z,p_{g},p_{g^{-1}})$ (in Equation \eqref{eq:tQS}). If $r=1$, we can retrieve a known formula for colourless random walks. We note that Equation \eqref{eq:Spf} should be compared to Proposition 3.1 in Lehner \cite{lehner1999computing} where a related formula is derived in a self-adjoint case (it can be checked that $2z$ is the inverse of the diagonal term of the Green function $(I - \tcP)^{-1} (e,e)$ where $(I - \tcP)^{-1} $ is seen as an infinite matrix indexed by $G \times G$ with coefficients in $M_r(\bC)$).

\section{Computing drift and entropy}\label{section:computing}

For classical convolution random walks, the general formulas for the rate of escape and entropy given in Theorem \ref{thm:formula_drift_entropy} go back to the work of Furstenberg \cite{MR163345} and are well known. The proofs extend easily to coloured random walks, so we will not give further details. We refer to \cite{ledrappier2001some, kaimanovich1983random} for the case of free groups and \cite{jean2005random, mairesse2007random, MR2378433} for more general free products.

We now turn to the derivation of \eqref{thm:formula_entropy_explicit} from the integral formula \eqref{eq:entropy_integral}. The latter shows that the computation of the entropy ultimately comes down to the computation of some Radon-Nikodym derivatives. In the colourless case, this computation goes basically as follows. Consider a cylinder $\xi_{1} \cdots \xi_{n}$. Compute $g \cdot p^{\infty}(\xi_{1} \cdots \xi_{n})$ distinguishing the different cases occurring:
\begin{itemize}
 \item if $g = \xi_{1}$ then $g \cdot p^{\infty}(\xi_{1} \cdots \xi_{n}) = p^{\infty}(\xi_{2} \cdots \xi_{n})$
 \item if $\xi_{1} \in G_{i}, g \neq \xi_{1}$ and there exists $h \in G_{i}$ such that $gh = \xi_{1}$, then $g \cdot p^{\infty}(\xi_{1} \cdots \xi_{n}) = p^{\infty}(h \xi_{2} \cdots \xi_{n})$
 \item otherwise, $g \cdot p^{\infty}(\xi_{1} \cdots \xi_{n}) = p^{\infty}(g^{-1} \xi_{1} \cdots \xi_{n}).$ 
\end{itemize}
As one can check, the last case occurs if and only if $\xi_{1} \in \Next(g^{-1})$. Expand now the expressions above into products using from \eqref{eq:harmonic2} that $p^{\infty}(\xi_{1} \cdots \xi_{n}) = q_{\xi_{1}} \cdots q_{\xi_{n-1}} \mu_{\xi_{n}}$ to get 
\[
 \frac{g \cdot p^{\infty}(\xi_{1} \cdots \xi_{n})}{p^{\infty}(\xi_{1} \cdots \xi_{n})} = 
 \left\{ \begin{array}{l l}
  1/q_{g} & \text{if $\xi_{1} = g$,}   \\                                    
  q_{g^{-1}} & \text{if $\xi_{1} \in \Next(g^{-1})$,} \\
  q_{h}/q_{\xi_{1}} = q_{g^{-1} \xi_{1}}/q_{\xi_{1}} & \text{if $h \in S$ and $gh = \xi_{1}$.}
\end{array}\right.
\]
%As will be proved after, the left-hand side converges to the Radon-Nikodym derivative as $n \rightarrow \infty$. 
One can then easily express the integral \eqref{eq:entropy_integral} in terms of the scalars $p_{g}, q_{g}$ and $\mu_{g}$ and obtain the formulas given in \cite{ledrappier2001some,mairesse2007random}. 

For coloured random walks, the same computations can be made with matrices except that the cancellation between products no longer takes place. Instead, one has to deal with an infinite product of random matrices. For example, as one might guess from \eqref{eq:defgxi}-\eqref{eq:harmonic2}, it is true that
\begin{equation}\label{eq:limit_radon_nikodym}
 \frac{d g p^{\infty}_{v}}{d p^{\infty}_{u}}(\xi) = \lim_{n \rightarrow \infty} \frac{ \II_{v}^{\top} \ q_{g^{-1}} q_{\xi{1}} \cdots q_{\xi{n-1}} \mu_{\xi{n}} \II}{\II_{u}^{\top} \  q_{\xi{1}} \cdots q_{\xi{n-1}} \mu_{\xi{n}} \II} \quad \text{a.s.}, 
\end{equation}
if $\xi_{1} \in \Next(g^{-1}) $. Convergence as above is the object of the next subsection.

\subsection{Convergence in direction for inhomogeneous products of matrices }

In this final subsection, we prove Theorem \ref{thm:formula_entropy_explicit}. The computation of the Radon-Nikodym derivatives in the integral formula of the entropy requires to investigate infinite products of random, non-negative matrices (that is, with non-negative entries). We use first use results for non-negative, deterministic matrices \cite{seneta2006non} to justify limits of infinite matrix products. 
This yields formulas like \eqref{eq:limit_radon_nikodym} and a first expression for the entropy \eqref{eq:radon_nikodym}. 
Then we use results for products of random matrices \cite{bougerol2012products, carmona2012spectral} to describe the law of the limits obtained, which yields eventually Theorem \ref{thm:formula_entropy_explicit}.

\paragraph{Convergence of the Radon-Nikodym derivatives.}
Instead of the general results in \cite{seneta2006non}, it turns out that we can directly use results from \cite{MR388545, lalley1993finite} which already involved matrix products to study finite range random walks. It does not seem however that the methods used in these references can be interpreted as an application of the linearization trick.

Let $X$ be a finite set and $Y$ a subset of $X \times X$ such that for all $x \in X$, the set $\{y \in X, (x,y) \in Y \}$ is non empty. Let $\Sigma := X^{\bZ}$ be the space of doubly infinite sequences $(\xi_{n})_{n \in \bZ}$ with values in $X$, such that $(\xi_{n},\xi_{n+1}) \in Y$ for all $n \in \bZ$. Let $\sigma : (\xi_{n})_{n \in \bZ} \mapsto (\xi_{n+1})_{n \in \bZ}$ denote the standard shift on $\Sigma$ and given a function $f : \Sigma \rightarrow \bC$ and $n \in \bN$, write 
\[
 S_{n}f := f + f \circ \sigma + \ldots f \circ \sigma^{n-1}.
\]
Recall the definition of $\cP^+$ in \eqref{eq:defcP+}.  We will use the following formulation of a result of \cite{lalley1993finite} on convergence of matrix products, but note that a similar result was already used in the earlier reference \cite{MR388545}.
 
\begin{prop}[{\cite[Prop 5.2]{lalley1993finite}}]\label{prop:convergence_product}
 Let $(M_{x})_{x \in X}$ be a family of $r \times r$ matrices with non-negative entries. Assume there exists integers $m \geq 0, k \geq 1$ and a function $B : X^{k} \rightarrow 2^{[r]} \smallsetminus \{ \emptyset \}$ (the set of non-empty subsets of $[r]$), such that for every $n \geq m$ and every family $x_{1}, \ldots, x_{n+k} \in X$ with $(x_{i},x_{i+1}) \in Y$ for all $i$, 
 \begin{equation}\label{eq:condition_imprimitive}
  (M_{x_{1}} M_{x_{2}} \cdots \ M_{x_{n+k}})_{u,v} > 0 \Leftrightarrow v \in B(x_{n+1}, \ldots, x_{n+k}) \qquad \forall u,v \in [r]
 \end{equation}
 Then there exist constants $C > 0$ and $0 < \alpha < 1$, maps $\varphi, \gamma : \Sigma \rightarrow \bR$ and $V,W : \Sigma \rightarrow \cPr$ such that for all $\xi \in \Sigma$,
 \begin{equation}\label{eq:convergence_product}
  \norm{e^{- S_{n} \varphi (\xi)} M_{\xi_{1}} M_{\xi_{2}} \cdots \ M_{\xi_{n}} - \gamma(\sigma_{n} \xi) V(\xi) W(\sigma^{n} \xi)^{\top}} \leq C \alpha^{n} ,
 \end{equation}
where $V=V(\xi_{1}, \xi_{2}, \ldots)$ depends only on the ``forward coordinates``, while $W = W(\xi_{0}, \xi_{-1}, \ldots)$ depends only on the ''backward coordinates''. Furthermore,
 \begin{equation}\label{eq:invariance_V}
  M_{\xi_{1}} V(\sigma \xi) = e^{\varphi(\xi)} V(\xi).
 \end{equation}
\end{prop}

%Propositions 5.1 and 5.2 in \cite{lalley1993finite} give additional relations relating the functions $\varphi, \gamma$ with the vectors $V,W$. They are not stated here as they will not be necessary for what follows.
Proposition \ref{prop:convergence_product} states shows under condition \eqref{eq:condition_imprimitive}, the product $M_{\xi_{1}} M_{\xi_{2}} \cdots \ M_{\xi_{n}}$ tends up to renormalization factor to a rank one matrix whose range is spanned by the vector $V(\xi)$.  
This result is well known for powers of a matrix with positive entries, in which case the vector $V$ is nothing but the Perron-Frobenius eigenvector.
Proposition \ref{prop:convergence_product} is thus a generalization of the Perron Frobenius theory to inhomogeneous products of non-negative matrices with possibly zero columns (condition \eqref{eq:condition_imprimitive}). We refer to \cite{lalley1993finite} for additional properties  of the functions $\varphi, \gamma$ with the vectors $V,W$ which will not be needed here.

To apply Proposition \ref{prop:convergence_product} in context of linearized random walks, take $X = S$ and $(x,y) \in Y$ if and only if $y \in \Next(g)$.

\begin{coroll}\label{cor:convergenceRN}
 Suppose the matrices $(q_{g})_{g \in S}$ satisfy the hypothesis of Proposition \ref{prop:convergence_product} and let $V: \partial G \rightarrow \cPr$ be the corresponding map. Then for all $u,v \in [r]$, $\xi \in \partial G$,
 \begin{equation}\label{eq:radon_nikodym}
  \frac{d g p^{\infty}_{v}}{d p^{\infty}_{u}}( \xi) = \left\{ \begin{array}{l l}
   \frac{\innerprod{\mathbbm{1}_{v}}{V(\sigma \xi) }}{\innerprod{q_{g}(u,\cdot)}{V(\sigma \xi)}} & \text{if $\xi_{1} = g$,}  \vspace{4pt} \\
   \frac{\langle q_{g^{-1}}(v,\cdot) , V(\xi) \rangle }{\innerprod{\mathbbm{1}_{u}}{V(\xi)}} & \text{if $\xi_{1} \in \Next(g^{-1})$,}   \vspace{4pt}\\
 \frac{\innerprod{q_{h}(v, \cdot)}{V(\sigma \xi)}}{\innerprod{q_{gh}(u, \cdot)}{V(\sigma \xi)}} & \text{if $h \in S$ and $gh = \xi_{1}$.}
\end{array}\right.
 \end{equation}
\end{coroll}

\begin{proof}
From \eqref{eq:defgxi}-\eqref{eq:harmonic2}, for any cylinder $\xi_{1} \cdots \xi_{n}$ 
 \[
  g \cdot p^{\infty}_{v}(\xi_{1} \cdots \xi_{n}) = \left\{ \begin{array}{l l}
     \II_{v}^{\intercal} q_{\xi_{2}} \cdots q_{\xi_{n-1}} \mu_{\xi_{n}} \II & \text{if $\xi_{1} = g$,} \\ 
     \II_{v}^{\intercal} q_{g^{-1}} q_{\xi_{1}} \cdots q_{\xi_{n-1}} \mu_{\xi_{n}} \II & \text{if $\xi_{1} \in \Next(g^{-1})$,} \\
     \II_{v}^{\intercal} q_{h} q_{\xi_{2}} \cdots \ \mu_{\xi_{n}} \II & \text{if $h \in S$ and $gh =\xi_{1}$.}
\end{array}\right.
 \]
 Let $\xi \in \partial G$. Combining the previous computation with Proposition \ref{prop:convergence_product}, there exist $\lambda_{n} >0$, uniformly lower bounded, and $\alpha < 1$ such that if $\xi_{1} = g$, 
 \begin{align*}
  \frac{g \cdot p^{\infty}_{v}(\xi_{1} \cdots \xi_{n})}{p^{\infty}_{u}(\xi_{1} \cdots \xi_{n})} &= \frac{\II_{v}^{\top} \left( q_{\xi_{2}} \cdots q_{\xi_{n-1}} \right) \left( \mu_{\xi_{n}} \II \right)}{q_{g}(u, \cdot)^{\top} \left( q_{\xi_{2}} \cdots q_{\xi_{n-1}} \right) \left( \mu_{\xi_{n}} \II \right)} \\
  &= \frac{\II_{v}^{\top} V( \sigma \xi) \lambda_{n} + O(\alpha^{n})}{q_{g}(u, \cdot)^{\top} V(\sigma \xi) \lambda_{n} + O(\alpha^{n})} \\
  &\xrightarrow[n \rightarrow \infty]{} \frac{ \innerprod{\II_{v}}{V(\sigma \xi)}}{\innerprod{q_{g}(u, \cdot)}{V(\sigma \xi)}}.
 \end{align*}
 The other cases are treated similarly.
%  \begin{align*}
%   \frac{g \cdot p^{\infty}_{v}(\xi_{1} \cdots \xi_{n})}{p^{\infty}_{u}(\xi_{1} \cdots \xi_{n})} &= \frac{q_{g^{-1}}(v, \cdot) \left( q_{\xi_{1}} \cdots q_{\xi_{n-1}} \right) \left( \mu_{\xi_{n}} \II \right)}{\II_{u} \left( q_{\xi_{1}} \cdots q_{\xi_{n-1}} \right) \left( \mu_{\xi_{n}} \II \right)} \\
%   &= \frac{q_{g^{-1}}(v, \cdot) V(\xi) \lambda_{n} + O(\alpha^{n})}{\II_{u} V(\xi) \lambda_{n} + O(\alpha^{n})} \\
%   &\xrightarrow[n \rightarrow \infty]{} \frac{ \innerprod{q_{g^{-1}}(v, \cdot)}{V(\xi)}}{\innerprod{\II_{u}}{V(\xi)}}.
%  \end{align*}
%  

 On the other hand 
 \[
  \frac{g \cdot p^{\infty}_{v}(\xi_{1} \cdots \xi_{n})}{p^{\infty}_{u}(\xi_{1} \cdots \xi_{n})} = \int_{\partial G} \frac{d g p^{\infty}_{v}}{p^{\infty}_{u}} \frac{\II_{\xi_{1} \cdots \xi_{n}}}{p^{\infty}_{u}(\xi_{1} \cdots \xi_{n})} \ d p^{\infty}_{u}
 \]
but as $n \rightarrow \infty$ the measure $\frac{\II_{\xi_{1} \cdots \xi_{n}}}{p^{\infty}_{u}(\xi_{1} \cdots \xi_{n})} \ d p^{\infty}_{u}$ converges to the Dirac mass at $\xi$: for all integer $k$ and $\epsilon > 0$, the mass of all $k$-cylinders except the cylinder $\xi_{1} \cdots \ \xi_{k}$ is bounded by $\epsilon$ for $n$ large enough. Thus the above limits give indeed the Radon-Nikodym derivative. 
\end{proof}

The next proposition asserts that for the coloured random walks that we are mainly interested in, Corollary \ref{cor:convergenceRN} applies. 
\begin{prop}\label{prop:linlalley}
 Suppose the coloured random walk $(Y_{n})_{n \geq 0}$ is a linearized random walk obtained from a finite range random walk $(X_{n})_{n \geq 0}$ via one of the procedures presented in Section \ref{section:linearization}. If $(X_{n})_{n \geq 0}$ is irreducible, then the matrices $(q_{g})_{g \in S}$, satisfy the hypothesis of Proposition \ref{prop:convergence_product}.
\end{prop}

\begin{proof}  

	Extend the definition \eqref{eq:defqg} of matrices $q_{g}$ to the whole group and notice that if $g \in G$ writes as a reduced word $g = g_{1} \cdots \ g_{n}$, then $q_{g} = q_{g_{1}} \cdots \ q_{g_{n}}$. We recall that given the set of generators considered \eqref{eq:generators} representatives are uniquely defined. Thus proving \eqref{eq:condition_imprimitive} comes down to prove that for $\abs{g}$ large enough, $q_{g}(u,v) > 0$ if and only if $v$ is in some subset of $[r]$ that depends only on the last letters of $g$. This statement regarding the coloured random walk can be transferred to a statement about the initial walk as follows.
		
		Consider the constructions of the linearized random walks in Section \ref{section:linearization}: both are made so that if $X_{0} = e$ and $Y_{0} = (e,1)$, one can couple $(Y_{n})_{n \geq 0}$ and $(X_{n})_{n \geq 0}$ to have $X_{n} = Y_{\tau_{n}}$ where $\tau_{n}$ is the $n$-th return time at colour $1$.
		Furthermore, for all colour $v \in [r]$, there exists (at least one) $h \in K$ with representative $h=h_{1} \cdots h_{k} h_{k+1} \cdots h_{l}$ such that $\bP_{(e,1)} \sbra{Y_{k} = (h_{1} \cdots h_{k}, v)} > 0$. The main difference between the proposed linearization procedures is the number of such $h$ that one can associate to the colour $v$, but this has no consequence for what follows.
		
		Let $K(v)$ be the set of such elements $h$. For each $h = h_{1} \cdots h_{k} h_{k+1} \cdots h_{l} \in K(v)$, set $p(h) := h_{1} \cdots h_{k}$ and $s(h) := h_{k+1} \cdots h_{l}$. Then consider the following events: given $a,g \in G$, $h=h_1 \cdots h_l \in K$, if there exists $k \leq l$ such that $a h_1 \cdots h_k = g$, let 
		\[ 
		A_{a,h}(g) := \left\{ \exists n \geq 0 \: | \: X_{n} = a, X_{n+1} = ah \right\}
		\]
		and $A_{a,h}(g) = \emptyset $ otherwise, and let $B_{a,h}(g)$ the event that $A_{a,h}(g)$ occurs before all other events $A_{b,h'}(g)$, $(b,h') \neq (a,h)$.
		
		Suppose now $Y_{0}$ starts with colour $u$. By $G$-invariance of transition probabilities we can suppose that it is at $(e,u)$. Then the starting colour imposes the walk to transition towards $(s(h),1)$ for some $h \in K(u)$. Once at state $(s(h),1)$, $Y_{n}$ will visit $g$ if and only $X_n$, started at $s(h)$, makes a transition that goes through $g$, ie if and only if the event $A_{a,h'}(g)$ occurs for some $a \in G, h' \in K$. Thanks to the irreducibility of $X_{n}$ it is clear that such transitions occur, but the difficulty is to make sure that $Y_{n}$ can arrive for the first time at $g$ with the prescribed color $v$. This will occur with the events $B_{a,h'}(g)$ for $h' \in K(v)$ and $a = g p(h')^{-1}$. Given $x \in G$ $h \in K$, let $m_{g}(x,h) := \bP_{x}\sbra{B_{g p(h)^{-1},h}(g)}$. What precedes thus sums up in the lower bound
		\[ 
		q_{g}(u,v) \geq \bP_{(e,u)} \sbra{\tau_{(s(h),1)} < \infty} m_{g}(s(h), h') \quad \forall h \in K(u), h' \in K(v)
		\]
		where $\bP_{(e,u)} \sbra{\tau_{(s(h),1)} < \infty} > 0$ is the probability that $Y_n$ visits $(s(h),1)$. In particular, it now suffices to prove that for all $x,g \in G$, $h \in K$, the positivity of $m_{g}(x,h)$ depends only on $h$ and the last letters of $g$, provided $\abs{g}$ is large enough. This can be shown in the same way as Lemma 5.3 in \cite{lalley1993finite} where a similar result is proved. 
		
		Let $L$ be the largest length of elements in the support of $X_{n}$. 
		By $G$-invariance of transition probabilities $m_{g}(x,h) = m_{x^{-1}g}(e,h)$ so one can in particular suppose $x = e$. From the irreducibility of $X_n$, every element of the generating set $S$ can be reached by the random walk in a bounded number of steps $l$, because $S$ is finite. Therefore for $\abs{g} > 2 L + l$, any path of positive probability that goes from $e$ to $gp(h)^{-1}$ without passing through $g$ and then moves from $gp(h)^{-1}$ to $gp(h)^{-1} h = gs(h)$ can be extended with an initial loop between $e$ and any $s \in S$. From this we deduce that $m_{g}(e,h) > 0$ if and only if $m_{g}(s,h) = m_{s^{-1}g}(e,h) > 0$. Choosing $s$ as the first letter of $g$, we deduce that the positivity of $m_{g}(a,b)$ does not depend on it. Iterating this argument we obtain that it only depends on $h$ and the $2L + l$ last letters of $g$. \end{proof}

\paragraph{Law of Radon-Nikodym derivatives.}

The previous result is purely deterministic and does not take into account that the process $(X_{\infty}^{(n)},u_{n})_{n \geq 0}$ is a Markov chain. Our goal in this paragraph is to determine the law of the probability vector $V$ appearing in Corollary \ref{cor:convergenceRN}. The monograph \cite{bougerol2012products} (see also \cite{carmona2012spectral}) state results for products of iid random matrices which, as we shall see, apply to Markovian products as well.

Remploying the general framework of Proposition \ref{prop:convergence_product}, consider a Markov chain $(Z_{n})_{ n \geq 0}$ on the finite state space $X$ with transition matrix $Q$ and $Y:= \{(x,y) \in X \times X, Q(x,y) > 0\}$. Let $(M_{x})_{x \in X}$ be a family of $r \times r$ non-negative matrices, such that condition \eqref{eq:condition_imprimitive} holds. 

Consider the sequence defined by $Y_{n} := M_{Z_{n}}, n \geq 1$. By Proposition \ref{prop:convergence_product}, a.s. the product $Y_{1} \cdots Y_{n}$ is non-zero and converges in direction to a rank one matrix spanned by a random vector $V \in \cPr$. 
If one multiplies the product on the left by another matrix $M_{x}$, one still has a products of matrices $M_{x}$, so one can expect the law of $V$ to satisfy some invariance property.

For all non-negative matrix $A \in M_{r}(\bR)$ and all $z \in \cPr$, we define $A z$ to be the normalized image of $z$ (identified with a vector in $\bR^{r}$) that makes it a probability vector, provided $A z \neq 0$. The latter case will not be an issue: condition \eqref{eq:condition_imprimitive} implies in particular that matrices $M_{x}$ have no zero row, and vectors of $\cPr$ have no zero coordinate so $M_{x} z$ is well defined for all $x \in X, z \in \cPr$.

Since $Z_n$ is a Markov chain, the law of $V$ may depend on the starting state of $Z_{n}$. Hence it is natural to think of it as a coloured measure (colours being here the states $x \in X$).

\begin{Def}
A family  $\nu=(\nu_{x})_{x \in X}$ of probability measure on $\cPr$ is called a coloured measure. Let $Q \ast \nu$ denote the coloured measure defined by 
 \[
  \int f(z) \: d (Q \ast \nu)_{x}(z) = \sum_{y} \int f(M_{y} z) Q(x,y) \: d \nu_{y}(z)
 \]
for all bounded measurable function $f$ on $\cPr$ and $x \in X$. The coloured measure $\nu$ is said to be invariant if $Q \ast \nu = \nu$.
\end{Def}

\begin{lem}\label{le:nuinvV}
  Let $\nu_{x}$ be the law of $V$ when $Z_{n}$ is started at $x$, the coloured measure $(\nu_{x})_{x \in X}$ is the unique coloured measure which is invariant with respect to $Q$. 
\end{lem}

\begin{proof}
We use Markov property together with \eqref{eq:invariance_V} to obtain the invariance of the law of $V$. 
 
 The proof of uniqueness is similar to the proof for the harmonic measure in Proposition \ref{prop:ex+uni_stationary}. Consider another invariant coloured measure $(\rho_{x})_{x \in X}$.
 By invariance, for all bounded measurable function $f$ on $\cPr$ the sequence 
 \[
  M_{n} := \int f( Y_{1} \cdots \ Y_{n} w) \: d \rho_{Z_{n}}(w), \quad n \geq 0
 \]
 is a bounded martingale with respect to the filtration $\cF_{n} := \sigma( Z_{0}, Z_{1}, \ldots, Z_{n})$, so $M_{n}$ converges a.s. and in $L^{1}$. This being true for all bounded measurable function, this implies the measures $Y_{1} \cdots \ Y_{n} \ast \rho_{Z_{n}}$ converges weakly to a measure. 
 Now, because the product $Y_{1} \cdots Y_{n}$ converges in direction to $V$, the limit is necessarily the Dirac mass at $V$. 
 
 On the other hand, the martingale property with the $L^{1}$ convergence gives that for all bounded measurable functions $f$ on $\cPr$, for all $x \in X$,
 \[
 \bE_{x} \sbra{ f(V) } =   \bE_{x} \sbra{ \int f \: \delta_{V} } =  \int f \: d \rho_{x} .
 \]
 Therefore $\rho_{x}$ must be the law of $V$ when $Z_0 =x$, that is $\rho_x = \nu_{x}$, and is consequently unique.
\end{proof}

We are finally ready to prove Theorem \ref{thm:formula_entropy_explicit}.

\begin{proof}[Proof of Theorem \ref{thm:formula_entropy_explicit}]
We apply Lemma \ref{le:nuinvV} in the context of coloured random walks, with $X = S \times [r]$ and $Q((g,u)(h,v)) = \mu_{h}(u,v)$. Due to the form of the transition probabilities the law of $V$ is only indexed by colours. It remains to use Corollary \ref{cor:convergenceRN} in conjonction with the integral formula of the entropy given by Theorem \ref{thm:formula_drift_entropy}.
\end{proof}

\paragraph{Acknowledgements}
Part of this work was performed during a visit of both authors to
Kyoto University. Both authors gratefully acknowledge the
support of JSPS and Kyoto University. C.B. was supported by the research grant ANR-16-CE40-0024-01.  We thank Jean Mairesse for stimulating discussions and Vadim Kaimanovich for pointing us reference \cite{kaimanovich2005munchhausen}.

\end{document}